\documentclass[10pt,a4paper]{scrartcl}

\usepackage[utf8]{inputenc}
\usepackage[T1]{fontenc}
\usepackage{amsfonts,bm}
\usepackage{amsmath}
\usepackage{amssymb}
\usepackage{amsthm,lmodern}
\usepackage[english]{babel}
\usepackage{color}
\usepackage[left=3cm,right=3cm,bottom=3cm,top=3cm]{geometry}
\usepackage{graphics}
\usepackage{graphicx}
\usepackage{calc}
\usepackage{multicol}
\usepackage{multirow}
\usepackage{array}
\usepackage{paralist}
\usepackage{stmaryrd}
\usepackage{textcomp}
\usepackage{tikz}
\usetikzlibrary{arrows,backgrounds,calc,decorations.pathmorphing,decorations.text,patterns,plotmarks}
\usepackage{url}
\usepackage{xcolor}
\usepackage[all]{xy}
\usepackage[maxnames=5]{biblatex}

\bibliography{biblio.bib}

\newtheorem{definition}{Definition}[section]
\newtheorem{theorem}[definition]{Theorem}

\newtheorem{lemma}[definition]{Lemma}
\newtheorem{proposition}[definition]{Proposition}
\newtheorem{corollary}[definition]{Corollary}

\newtheoremstyle{remarks}% name of the style to be used
  {\topsep}% measure of space to leave above the theorem. E.g.: 3pt
  {\topsep}% measure of space to leave below the theorem. E.g.: 3pt
  {\rmfamily}% name of font to use in the body of the theorem
  {0pt}% measure of space to indent
  {\bfseries}% name of head font
  {. ---}% punctuation between head and body
  { }% space after theorem head; " " = normal interword space
  {\thmname{#1}\thmnumber{ #2}\thmnote{ (#3)}}
\theoremstyle{remarks}
\newtheorem{remark}[definition]{Remark}

\def\calli#1{\expandafter\def\csname
  #1\endcsname{\mathcal{#1}}}	% Calligraphic I -> \I
\def\sets#1{\expandafter\def\csname
  bb#1\endcsname{\mathbb{#1}}}	% Blackboard Bold R -> \bbR
\def\rebar#1{\expandafter\def\csname #1bar\endcsname{\overline{\csname
      #1\endcsname}}}		% Overline w -> \wbar
\def\gothify#1{\expandafter\def\csname
  #1#1#1\endcsname{\mathfrak{#1}}}	% Gothic F -> \FFF

\sets{R}	% \bbR
\sets{N}	% \bbN
\sets{Z}	% \bbZ
\calli{I}	% \I
\calli{A}	% \A
\calli{M}	% \M
\calli{T}	% \T
\calli{R}
\calli{U}
\calli{W}
\gothify{F}	% \FFF
\gothify{G}	% \GGG
\rebar{M}	% \Mbar

\renewcommand{\SS}{\mathcal{S}}

\newcommand{\Up}[1]{\,\Uparrow #1}
\newcommand{\upp}[1]{\,\uparrow #1}

\newcommand{\slgb}{\mbox{$\sigma$-al}\-ge\-bra}

\newcommand{\CDOT}{\!\!\!\;\cdot\!\!\!\;}
\newcommand{\tq}{\;:\;}

\newcommand{\un}[1]{\mathbf{1}\bm(#1\bm)}

\newcommand{\height}{\tau}
\newcommand{\br}[1]{{B}^{+}_{#1}}
\newcommand{\gbr}[1]{{B}_{#1}}
\newcommand{\dbr}[1]{{B}^{+*}_{#1}}
\newcommand{\bbr}[1]{{B}^{?}_{#1}}

\newcommand{\bbrbar}[1]{\overline{\bbr{#1}}}

\newcommand{\bdbr}[1]{\partial\br{#1}}
\newcommand{\bdbbr}[1]{\partial\bbr{#1}}
\newcommand{\unit}{\textbf{e}}
\newcommand{\leql}{\leq_\text{l}}
\newcommand{\leqr}{\leq_{\text{r}}}
\DeclareMathOperator{\bigveel}{\mathchoice%
  {\bigvee{}_{\raisebox{-.8ex}{\scriptsize\!\!\text{l}}}}%
  {\bigvee{}_{\raisebox{-.3ex}{\scriptsize\!\!\text{l}}}}%
  {\bigvee{}_{\raisebox{-.17ex}{\tiny\!\text{l}}}}%
  {\bigvee{}_{\raisebox{-.17ex}{\tiny\!\text{l}}}}%
}
\DeclareMathOperator{\bigveer}{\mathchoice%
  {\bigvee{}_{\raisebox{-.8ex}{\scriptsize\!\!\text{r}}}}%
  {\bigvee{}_{\raisebox{-.3ex}{\scriptsize\!\!\text{r}}}}%
  {\bigvee{}_{\raisebox{-.17ex}{\tiny\!\text{r}}}}%
  {\bigvee{}_{\raisebox{-.17ex}{\tiny\!\text{r}}}}%
}
\DeclareMathOperator{\bigwedgel}{\mathchoice%
  {\bigwedge{}_{\raisebox{-.8ex}{\scriptsize\!\text{l}}}}%
  {\bigwedge{}_{\raisebox{-.3ex}{\scriptsize\!\text{l}}}}%
  {\bigwedge{}_{\raisebox{-.17ex}{\tiny\!\text{l}}}}%
  {\bigwedge{}_{\raisebox{-.17ex}{\tiny\!\text{l}}}}%
}
\newcommand{\veel}{\vee_{\!\text{l}}}
\newcommand{\wedgel}{\wedge_{\text{l}}}

\newcommand{\Dl}{\mathsf{L}}
\newcommand{\Dr}{\mathsf{R}}
\newcommand{\DD}{\mathcal{D}}
\newcommand{\CC}{\mathcal{C}}

\newcommand{\NN}{\mathcal{N}}

\newcommand{\abs}[1]{|#1|}
\newcommand{\rc}{\theta}
\newlength{\toplength}
\newcommand{\mybox}{$-1/2+3\rc$}
\newcolumntype{C}{>{\centering$}p{\toplength}<{$}}
\newcolumntype{D}{c}

\numberwithin{equation}{section}

\title{Uniform measures on braid monoids\\
and dual braid monoids} \date{}
\author{Samy Abbes\and Sébastien Gouëzel\and
  Vincent Jugé\and Jean Mairesse}

\begin{document}

\maketitle

\begin{abstract}
\begin{center}
\textbf{Abstract}
\end{center}\smallskip
We aim at studying the asymptotic properties of typical \emph{positive
  braids}, respectively \emph{positive dual braids}.  Denoting by
$\mu_k$ the uniform distribution on positive (dual) braids of
length~$k$, we prove that the sequence $(\mu_k)_k$ converges to a
unique probability measure $\mu_{\infty}$ on \emph{infinite} positive
(dual) braids. The key point is that the limiting measure
$\mu_{\infty}$ has a Markovian structure which can be described
explicitly using the combinatorial properties of braids encapsulated
in the Möbius polynomial. As a by-product, we settle a conjecture by
Gebhardt and Tawn (J.~Algebra, 2014) on the shape of the Garside
normal form of large uniform braids.
%
% We introduce the notion of uniform measure for positive braid monoids and
% for positive dual braid monoids (the later are also called Birman-Ko-Lee
% monoids). Uniform measures reflect the combinatorics of braid monoids on a
% probabilistic level, by giving illuminating probabilistic interpretations
% to two combinatorial devices: first, the Möbius polynomial; and second, the
% finite graph associated with the Garside structure of the monoid---its
% strong connectivity for dual braid monoids is an original result.

% Among all uniform measures on a braid monoid, either standard or dual, one
% and only one is not discrete, it is the \emph{uniform measure at infinity}
% for braids. Its characterization yields applications in the asymptotic
% study of finite uniform distributions on braids, including the settlement
% of a conjecture issued in 2014 by Gebhardt and Tawn.

\end{abstract}

\section{Introduction}
\label{sec:introduction}

Consider a given number of strands, say $n$, and the associated positive
braid monoid $\br n$  defined by the following monoid presentation, known as
the {\em Artin} presentation:
\begin{gather}
\label{eq:0*}
\arraycolsep=0.3pt
\br n = \left\langle \sigma_1,\ldots,\sigma_{n-1} \left|
\begin{array}{ll} \sigma_i \sigma_j = \sigma_j \sigma_i &\quad \text{for $|i-j| \geq 2$} \\
\sigma_i \sigma_{j} \sigma_i = \sigma_{j} \sigma_i \sigma_{j} &\quad \text{for $|i-j|=1$}
\end{array} \right.\right\rangle^+\,.%\qquad\text{(presented \emph{group})}
\end{gather}
The elements of $\br n$, the {\em positive braids}, are therefore equivalence
classes of words over the alphabet $\Sigma= \{\sigma_1,\dots,
\sigma_{n-1}\}$. Alternatively, going back to the original geometric
intuition, positive braids can be viewed as isotopy classes of {\em positive}
braid diagrams, that is, braid diagrams in which the bottom strand always
goes on top in a crossing, see Figure~\ref{fig:braidsojf}.

\begin{figure}
  \begin{align*}
\begin{tikzpicture}
\draw[thick,green] (0,0) -- (.5,0);	% first strand
\draw[thick,green] (0.5,0) -- (1,.5);
\draw[thick,green] (1,.5) -- (1.25,.5);
\draw[thick,green] (1.25,0.5) -- (1.75,1);
\draw[thick,green] (1.75,1) -- (2,1);
\draw[thick,green] (2,1) -- (2.5,1.5);
\draw[thick,green] (2.5,1.5) -- (3.75,1.5);
\draw[thick,red] (0,.5) -- (.5,.5);	% second strand
\draw[thick,red] (.5,.5) -- (.7,.3);
\draw[thick,red] (.8,.2) -- (1,0);
\draw[thick,red] (1,0) -- (3.75,0);
\draw[thick,blue] (0,1) -- (1.25,1);	% third strand
\draw[thick,blue] (1.25,1) -- (1.45,.8);
\draw[thick,blue] (1.55,.7) -- (1.75,.5);
\draw[thick,blue] (1.75,.5) -- (2.75,.5);
\draw[thick,blue] (2.75,.5) -- (3.25,1);
\draw[thick,blue] (3.25,1) -- (3.75,1);
\draw[thick,orange] (0,1.5) -- (2,1.5);	% fourth strand
\draw[thick,orange] (2,1.5) -- (2.2,1.3);
\draw[thick,orange] (2.3,1.2) -- (2.5,1);
\draw[thick,orange] (2.5,1) -- (2.75,1);
\draw[thick,orange] (2.75,1) -- (2.95,.8);
\draw[thick,orange] (3.05,.7) -- (3.25,.5);
\draw[thick,orange] (3.25,.5) -- (3.75,.5);
\node at (-.5,0){$1$};
\node at (-.5,.5){$2$};
\node at (-.5,1){$3$};
\node at (-.5,1.5){$4$};
\node at (0.75,-.5){$\sigma_1$};
\node at (1.5,-.5){$\sigma_2$};
\node at (2.25,-.5){$\sigma_3$};
\node at (3,-.5){$\sigma_2$};
\end{tikzpicture}
    &&
\begin{tikzpicture}
\node at (-.5,0){$1$};
\node at (-.5,.5){$2$};
\node at (-.5,1){$3$};
\node at (-.5,1.5){$4$};
\node at (0.75,-.5){$\sigma_3$};
\node at (1.5,-.5){$\sigma_1$};
\node at (2.25,-.5){$\sigma_2$};
\node at (3,-.5){$\sigma_3$};
\draw[thick,green] (0,0) -- (1.25,0);	% first strand
\draw[thick,green] (1.25,0) -- (1.75,.5);
\draw[thick,green] (1.75,.5) -- (2,.5);
\draw[thick,green] (2,0.5) -- (2.5,1);
\draw[thick,green] (2.5,1) -- (2.75,1);
\draw[thick,green] (2.75,1) -- (3.25,1.5);
\draw[thick,green] (3.25,1.5) -- (3.75,1.5);
\draw[thick,red] (0,.5) -- (1.25,.5);	% second strand
\draw[thick,red] (1.25,.5) -- (1.45,.3);
\draw[thick,red] (1.55,.2) -- (1.75,0);
\draw[thick,red] (1.75,0) -- (3.75,0);
\draw[thick,blue] (0,1) -- (.5,1);	% third strand
\draw[thick,blue] (.5,1) -- (1,1.5);
\draw[thick,blue] (1,1.5) -- (2.75,1.5);
\draw[thick,blue] (2.75,1.5) -- (2.95,1.3);
\draw[thick,blue] (3.05,1.2) -- (3.25,1);
\draw[thick,blue] (3.25,1) -- (3.75,1);
\draw[thick,orange] (0,1.5) -- (0.5,1.5);	% fourth strand
\draw[thick,orange] (0.5,1.5) -- (0.7,1.3);
\draw[thick,orange] (0.8,1.2) -- (1,1);
\draw[thick,orange] (1,1) -- (2,1);
\draw[thick,orange] (2,1) -- (2.2,.8);
\draw[thick,orange] (2.3,.7) -- (2.5,.5);
\draw[thick,orange] (2.5,.5) -- (3.75,.5);
\end{tikzpicture}
  \end{align*}
  \caption{Two isotopic braid diagrams representing a positive braid on $4$
    strands. Left: diagram corresponding to the word
    $\sigma_1\sigma_2\sigma_3\sigma_2$\,.  Right: diagram
    corresponding to the word
    $\sigma_3\sigma_1\sigma_2\sigma_3$\,. }
%The two words are congruent since
%    $\sigma_1\sigma_2\sigma_3\sigma_2\equiv\sigma_1\sigma_3\sigma_2\sigma_3\equiv\sigma_3\sigma_1\sigma_2\sigma_3$\,.}
  \label{fig:braidsojf}
\end{figure}
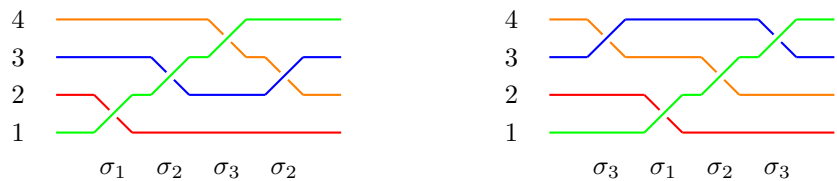

\medskip

We want to address the following question:
\begin{quotation}\label{quote}
What does a typical complicated positive braid look like?
\end{quotation}

To make the question more precise, we need to clarify the meaning of
``complicated'' and ``typical''.  First, let the complexity of a positive
braid be measured by the length (number of letters) of any representative
word. This is natural since it corresponds to the number of crossings between
strings in any representative braid diagram. Therefore, a positive braid is
``complicated'' if its length is large.

Second, let us  define a ``typical'' braid as a braid being picked at random
according to some probability measure. The two natural candidates for such a
probability measure are as follows. Fix a positive integer $k$.
\begin{itemize}
\item The first option consists in running a simple random walk on
  $\br n$~: pick a sequence of random elements $x_i, i\geq 1,$
  independently and uniformly among the generators
  $\Sigma= \{\sigma_1,\dots , \sigma_{n-1}\}$, and consider the
  ``typical'' braid $X=x_1\cdot x_2\cdot \dots \cdot x_k$. It
  corresponds to drawing a word uniformly in~$\Sigma^k$ and then
  considering the braid it induces.
\item The second option consists in picking a ``typical'' braid of length
    $k$ uniformly at random among all braids of length~$k$.
\end{itemize}
The two approaches differ since the number of representative words varies
among positive braids of the same length. For instance, in $\br 3$ and for
the length 3, the braid $\sigma_1\cdot \sigma_2 \cdot \sigma_1$ $(=
\sigma_2\cdot \sigma_1 \cdot \sigma_2)$ will be picked with probability 2/8
in the first approach, and with probability 1/7 in the second one, while all
the other braids of length 3 will be picked respectively with probabilities
1/8 and 1/7 in the two approaches. The random walk approach has been studied
for instance in~\cite{MaMa06,vershik00}; it is a special instance of random
walks on (semi)groups, see~\cite{woes}. In this paper, our focus is on the
second approach, that is, on {\em uniform measures on positive braids}.

\medskip

Let $\mu_k$ be the uniform probability measure on positive braids of
$\br n$ of length~$k$. The general question stated above can now be
rephrased as follows : study $\mu_k$ for large~$k$. Let us say that we
are interested in some specific property, say, the number of
occurrences of the Garside element $\Delta$ in a large random
braid. To study it, a first approach consists in performing a
numerical evaluation. To that purpose, the key ingredient is to have a
{\em sampling algorithm}, that is, a random procedure which takes as
input $k$ and returns as output a random braid of
distribution~$\mu_k$. Another, more intrinsic, approach consists first
in defining a probability measure $\mu_{\infty}$ on {\em infinite}
positive braids, encapsulating all the measures~$\mu_k$, and then in
studying the asymptotics of the property
\textit{via}~$\mu_{\infty}$. None of these two paths is easy to
follow. The difficulty is that the probability measures $(\mu_k)_k$
are not consistent with one another. For instance, in~$\br 3$, we
have:
\begin{equation}\label{non-consistent}
1/4 = \mu_2(\sigma_1\cdot \sigma_1) \neq \mu_3(\sigma_1\cdot \sigma_1\cdot \sigma_1)
  +\mu_3(\sigma_1\cdot \sigma_1\cdot \sigma_2)= 2/7 \:.
\end{equation}
Therefore, there is no obvious way to design a dynamic process to sample
braids. As another consequence, the Kolmogorov consistency theorem does not
apply, and there is no simple way to define a uniform probability measure on
infinite positive braids. This is in sharp contrast with the simpler picture
for the random walk approach described above.

\medskip

To overcome the difficulties, the rich combinatorics of positive braids has
to enter the scene. Going back to Garside~\cite{garside1969braid} and
Thurston~\cite{ECHLPT}, it is known that positive braids admit a \emph{normal
form}, that is a selection of a unique representative word for each braid,
which is \emph{regular}, that is recognized by a finite automaton. This so
called \emph{Garside normal form} enables to count positive braids in an
effective way, see for instance Brazil~\cite{braz}, but a non-efficient one
since the associated automaton has a large number of states, exponential in
the number of strands $n$, see Dehornoy~\cite{dehornoy07}. A breakthrough is
provided by Bronfman~\cite{bronfman01} (see also~\cite{albenque09}) who
obtains,  using an inclusion-exclusion principle, a simple recursive formula
for counting positive braids. Based  on this formula, a sampling algorithm
whose time and space complexities are polynomial in both the number of
strands $n$ and the length $k$ is proposed by Gebhardt and Gonzales-Meneses
in~\cite{gebhard13}. Using the sampling procedure, extensive numerical
evaluations are performed by Gebhardt and Tawn in~\cite{gebhardt14}, leading
to the {\em stable region conjecture} on the shape of the Garside normal form
of large uniform braids.

In the present paper, we complete the picture by proving the existence of a
natural uniform probability measure $\mu_{\infty}$ on infinite positive
braids. The measure induced by $\mu_{\infty}$ on braids of length $k$ is not
equal to~$\mu_k$, which is in line with the non-consistency illustrated
in~(\ref{non-consistent}), but the sequence $(\mu_k)_k$ does converge weakly
to~$\mu_{\infty}$. The remarkable point is that the measure $\mu_{\infty}$
has a Markovian structure which can be described explicitly. It makes it
possible to get precise information on $\mu_k$ for large $k$ by using the
limit~$\mu_{\infty}$. For instance,  we prove that the number of $\Delta$ in
a random braid of $\br n$ is asymptotically geometric of parameter
$q^{n(n-1)/2}$ where $q$ is the unique root of smallest modulus of the Möbius
polynomial of~$\br n$. As another by-product of our results, we settle the
stable region conjecture, proving one of the two statements in the
conjecture, and refuting the other one. Our different results are achieved by
strongly relying on refined properties of the combinatorics of positive
braids, some of them new.

\medskip

\textit{Mutatis mutandis}, the results also hold in the Birman-Ko-Lee dual
braid monoid~\cite{birman1998new}. We present the results in a unified way,
with notations and conventions allowing to cover the braids and the dual
braids at the same time. The prerequisites on these two monoids are recalled
in Section~\ref{sec:posit-dual-posit}, and the needed results on the
combinatorics of braids are presented in
Section~\ref{sec:garside-normal-form}. The main results are proved in
Section~\ref{sec:unif-meas-braid}, with applications in
Section~\ref{sec:appl-asympt-finite}, including the clarification of the
stable region conjecture. In Section~\ref{se-explicit}, we provide explicit
computations of the uniform measure $\mu_{\infty}$ for the braid monoid and
the dual braid monoid on 4 strands.  At last, analogs and extensions are
evoked in Section~\ref{se-ext}. Indeed, our results on braid monoids form a
counterpart to the results on trace monoids in~\cite{abbes15a,abbes15b}, and,
in a forthcoming paper~\cite{opus2}, we plan to prove results in the same
spirit for Artin-Tits monoids, a family encompassing both braids and traces.

%Authors usually distinguish between \emph{braid groups} and
%various submonoids of braid groups, such as
%\emph{positive braid monoids} and \emph{positive dual braid monoids};
%being only concerned by such monoids, we
%will simply speak of \emph{braids}, of \emph{braid monoids} and of \emph{dual braid monoids}.

\section{Positive and dual positive braid monoids}
\label{sec:posit-dual-posit}

In this section we introduce some basics on the monoid of positive braids and
the monoid of positive dual braids. We recall the notions of simple braids
for these monoids, as well of combinatorial representations of them.

\subsection{Two distinct braid monoids}
\label{sec:posit-braid-mono}

\subsubsection{The braid group and two of its submonoids}
\label{sec:presentations}

For each integer $n\geq2$, the \emph{braid group} $\gbr n$ is the group with
the following group presentation:
\begin{gather}
\label{eq:1*}
\arraycolsep=0.3pt
\gbr n = \left\langle \sigma_1,\ldots,\sigma_{n-1} \left|
\begin{array}{ll} \sigma_i \sigma_j = \sigma_j \sigma_i &\quad \text{for $|i-j| \geq 2$} \\
\sigma_i \sigma_{j} \sigma_i = \sigma_{j} \sigma_i \sigma_{j} &\quad \text{for $|i-j|=1$}
\end{array} \right.\right\rangle\,.%\qquad\text{(presented \emph{group})}
\end{gather}

Elements of $\gbr n$ are called \emph{braids}. Let $\unit$ and
``$\cdot$'' denote respectively the unit element and the concatenation
operation in~$\gbr n$\,. It is well known since the work of Artin that
elements of $\gbr n$ correspond to isotopy classes of braid diagrams
on $n$ strands, as illustrated in Figure~\ref{fig:braidsojf}; the
elementary move where strand $i$ crosses strand $i+1$ from above
corresponds to generator~$\sigma_i$, and the move where strand $i$
crosses strand $i+1$ from behind to~$\sigma_i^{-1}$\,.

We will be interested in two submonoids of~$\gbr n$\,.  The \emph{positive
braids monoid} $\br n$ is the submonoid of $\gbr n$ generated by
$\{\sigma_1,\ldots,\sigma_{n-1}\}$\,; and the \emph{positive dual braid
monoid} $\dbr n$ is the submonoid of $\gbr n$ generated by $\{\sigma_{i,j}\
|\ 1\leq i<j\leq n\}$\,, where $\sigma_{i,j}$ is defined by:
\begin{align*}
\sigma_{i,j}&=\sigma_i\,,&&\text{for $1\leq i<n$ and $j=i+1$,}\\
\sigma_{i,j}&=\sigma_i \sigma_{i+1} \ldots \sigma_{j-1}
\sigma_{j-2}^{-1} \sigma_{j-3}^{-1} \ldots
\sigma_i^{-1}\,,&&\text{for $1\leq i< n-1$ and $i+2\leq j\leq n$}\,.
\end{align*}
Observe the inclusion $\br n\subseteq\dbr n$, since each generator $\sigma_i$
of $\br n$ belongs to~$\dbr n$. The elements $\sigma_{i,j}$ are often called
\emph{Birman-Ko-Lee generators} in the literature, while the elements
$\sigma_i$ are called \emph{Artin} generators.

\paragraph*{Running examples for $n=3$.}
\label{sec:runn-exampl-n=3-1}

Throughout the paper, we shall illustrate the notions and results on the most
simple, yet non trivial  examples of braid monoids, namely on $\br 3$ and
on~$\dbr 3$\,:
\begin{align*}
  \br 3&=\langle\sigma_1,\;\sigma_2\rangle^+\,,\\
\dbr 3&=\langle
        \sigma_{1,2},\;\sigma_{2,3},\;\sigma_{1,3}\rangle^+
&\text{with }&
               \begin{aligned}
\sigma_{1,2}&=\sigma_1\,,&\sigma_{2,3}&=\sigma_2\,,
&\sigma_{1,3}&=\sigma_1\cdot\sigma_2\cdot\sigma_1^{-1}\,.
               \end{aligned}
\end{align*}

\subsubsection{Presentations of the monoids}
\label{sec:presentations-1}

Defining $\br n$ and $\dbr n$ as submonoids of~$\gbr n$\,, as we just did, is
one way of introducing them. Another one is through generators and relations.

First, $\br n$~is isomorphic to the monoid with the monoid
presentation~(\ref{eq:0*}), that is, the same presentation as $B_n$ but
viewed as a monoid presentation instead of a group presentation. Second,
$\dbr n$ is isomorphic to the monoid with $n(n-1)/2$ generators
$\sigma_{i,j}$ for $1\leq i<j\leq n$ and the following relations, provided
that the convention $\sigma_{j,i}=\sigma_{i,j}$ for $i<j$ is in force:
\begin{gather}
  \label{eq:3*}
  \begin{cases}
\sigma_{i,j} \sigma_{j,k} = \sigma_{j,k} \sigma_{k,i} = \sigma_{k,i} \sigma_{i,j} &
\text{for $ 1 \leq i < j < k \leq n$} \\
\sigma_{i,j} \sigma_{k,\ell} = \sigma_{k,\ell} \sigma_{i,j} & \text{for $ 1 \leq i < j < k < \ell \leq n$} \\
\sigma_{i,j} \sigma_{k,\ell} = \sigma_{k,\ell} \sigma_{i,j} & \text{for $ 1 \leq i < k < \ell < j \leq n  $}  \:.
  \end{cases}
\end{gather}

%It is indeed well known
%\cite{garside1969braid} that, as \emph{monoids}, $\br n$~is isomorphic
%to the monoid with the same presentation~(\ref{eq:1*}) as the braid
%group, but as a monoid; elements of~$\br n$\,, called \emph{positive
%  braids}, correspond therefore to braid diagrams involving only
%crossing of strands in the same direction, as illustrated in
%Figure~\ref{fig:braidsojf}.
%
%As for~$\dbr n$, it is isomorphic to the monoid with $n(n-1)/2$
%generators $\sigma_{i,j}$ for $1\leq i<j\leq n$ and the following
%relations, provided that the convention $\sigma_{j,i}=\sigma_{i,j}$
%for $i<j$ is in force:
%\begin{gather}
%  \label{eq:3*}
%  \begin{cases}
%\sigma_{i,j} \sigma_{j,k} = \sigma_{j,k} \sigma_{k,i} = \sigma_{k,i} \sigma_{i,j} &
%\text{for $ 1 \leq i < j < k \leq n$} \\
%\sigma_{i,j} \sigma_{k,\ell} = \sigma_{k,\ell} \sigma_{i,j} & \text{for $ 1 \leq i < j < k < \ell \leq n$} \\
%\sigma_{i,j} \sigma_{k,\ell} = \sigma_{k,\ell} \sigma_{i,j} & \text{for $ 1 \leq i < k < \ell < j \leq n  $}
%  \end{cases}
%\end{gather}

Elements of~$\br n$ are called \emph{positive
  braids}, they correspond to isotopy classes of braid diagrams involving only
crossing of strands in the same direction, see Figure~\ref{fig:braidsojf}.
Elements of $\dbr n$ are called \emph{dual positive braids}. They correspond
to isotopy classes of chord diagrams~\cite{birman1998new}. This time, there
are still $n$ strands but they are arranged along a cylinder; the element
$\sigma_{i,j}$ corresponds to a crossing of strands $i$ and~$j$. See
Figure~\ref{fig:udalbraids}.

\begin{figure}
\centering
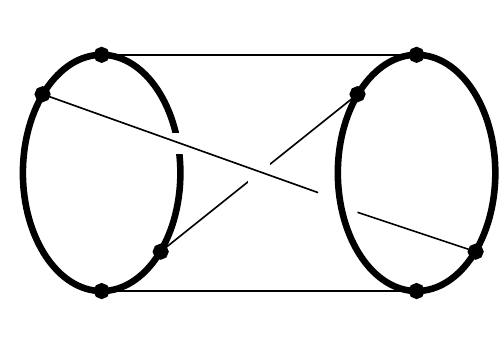
  \caption{Representation of $\sigma_{1,3}$ in $\dbr 4$}
  \label{fig:udalbraids}
\end{figure}

\medskip

The inclusion $\br n\subseteq\dbr n$ comes with the definition of
$\br n$ and $\dbr n$ as submonoids of the braid group $\gbr n$. It can
be obtained as follows when considering $\br n$ and $\dbr n$ as
abstract monoids with generators and relations. Let
$\iota:\{\sigma_1,\ldots,\sigma_n\}\to\dbr n$ be defined by
$\iota(\sigma_i)=\sigma_{i,i+1}$\,, and keep the notation $\iota$ to
denote the natural extension on the free monoid
$\iota:\{\sigma_1,\ldots,\sigma_n\}^*\to\dbr n$\,. It is clear that
$\iota$ is constant on congruence classes of positive braids, whence
$\iota$ factors through $\iota:\br n\to\dbr n$. It can then be proved
that this morphism is injective~\cite{garside1969braid,birman1998new}.

\begin{remark}
\label{rem:6} We emphasize that all the notions that we are about to define
on $\br n$ and on $\dbr n$ may or may not coincide on $\br n\cap\dbr n=\br
n$\,. Henceforth, it is probably clearer to keep in mind the point of view on
these monoids through generators and relations, rather than as submonoids
of~$\gbr n$\,.
\end{remark}

\paragraph*{Running examples for $n=3$.}
\label{sec:runn-exampl-n=3}

The presentations of the monoids $\br 3$ and $\dbr 3$ are the following:
\begin{align*}
  \br 3&=\langle
         \sigma_1,\ \sigma_2 \mid
         \sigma_1\sigma_2\sigma_1=\sigma_2\sigma_1\sigma_2\rangle^+\\
  \dbr 3&=\langle \sigma_{1,2},\ \sigma_{2,3},\ \sigma_{1,3} \mid \sigma_{1,2}\sigma_{2,3} =
          \sigma_{2,3}\sigma_{1,3} = \sigma_{1,3}\sigma_{1,2}\rangle^+
\end{align*}

\subsubsection{A common notation}
\label{sec:simultaneous}

We will consider simultaneously the monoids $\br n$ and~$\dbr n$. Henceforth,
we will denote by $\bbr n$ a monoid which, unless stated otherwise, may be
either the monoid $\br n$ or $\dbr n$.  The statements that we will prove for
the monoid $\bbr n$ will then hold for both monoids $\br n$ and~$\dbr n$.

In addition, we will denote by $\Sigma$ the set of generators of $\bbr n$,
hence $\Sigma = \{\sigma_i \tq 1 \leq i \leq n-1\}$ if $\bbr n = \br n$ and
$\Sigma = \{\sigma_{i,j} \tq 1 \leq i < j \leq n\}$ if $\bbr n = \dbr n$.

\subsubsection{Length and division relations. Mirror mapping}
\label{sec:length-divis-relat}

The above presentations~(\ref{eq:0*}) and~(\ref{eq:3*}) of $\bbr n$ are
homogeneous, meaning that the relations involve words of the same lengths.
Hence, the \emph{length} of $x\in\bbr n$\,, denoted by~$\abs{x}$, is the
length of any word in the equivalence class~$x$, with respect to the
congruence defining~$\bbr n$.

\begin{remark}
\label{rem:2} The length is an example of a quantity which is defined on both
$\br n$ and on~$\dbr n$, and which is invariant on~$\br n$\,. That is to say,
the length $\abs x$ of a positive braid $x\in\br n$ is invariant whether $x$
is considered as an element of $\br n$ or as an element of~$\dbr n$. Indeed,
if $x$ has length $k$ as an element of~$\br n$, then it necessarily writes as
a product $x=\sigma_{\varphi(1)}\cdot\ldots\cdot \sigma_{\varphi(k)}$ for
some function $\varphi:\{1,\ldots,k\}\to\{1,\ldots,n-1\}$\,. This entails
that~$x$, as an element of~$\dbr n$, writes as
$x=\sigma_{\varphi(1),\varphi(1)+1}\cdot\ldots\cdot\sigma_{\varphi(k),\varphi(k)+1}$\,,
and thus $x$ also has length $k$ as an element of~$\dbr n$\,.
\end{remark}

The monoid $\bbr n$ is equipped with the \emph{left} and with the \emph{right
divisibility} relations, denoted respectively $\leql$ and~$\leqr$\,, which
are both partial orders on $\bbr n$, and are defined by:
\begin{align*}
  x\leql y&\iff\exists z\in\bbr n\quad y=x\cdot z,&
  x\leqr y&\iff\exists z\in\bbr n\quad y=z\cdot x.&
\end{align*}

The mirror mapping, defined on words by $a_1\ldots a_k\mapsto a_k\ldots
a_1$\,, factorizes through $\bbr n$ and induces thus a \emph{mirror mapping}
on braids, denoted by $x\in\bbr n\mapsto x^*\in\bbr n$. It is an involutive
anti-isomorphism of monoids, it preserves the length of braids and swaps the
left and right divisibility relations:
\begin{align*}
  \forall x\in\bbr n\quad\abs{x^*}&=\abs{x}\,,&
\forall x,y\in\bbr n\quad x\leqr y&\iff x^*\leql y^*\,.
\end{align*}

The mirror mapping being an isomorphism of partial orders $(\bbr
n,\leql)\to(\bbr n,\leqr)$, we shall focus on the left divisibility relation
$\leql$ only.

\begin{remark}
  \label{rem:3}
  Following Remark~\ref{rem:2}, it is clear that the left divisibility
  is also invariant on~$\br n$\,: if $x,y\in\br n$ are such that
  $x\leql y$ in~$\br n$\,, then $x\leql y$ also holds in~$\dbr n$\,.
  Observe however that the converse is not true. For instance,
  consider the case $n=3$ and set $x=\sigma_2$ and $y=\sigma_1\cdot\sigma_2$.
In $\br 3$, clearly, $x\leql y$ does not
  hold. However, in $\dbr 3$, we have $x=\sigma_{2,3}$ and $y=\sigma_{1,2}\cdot\sigma_{2,3}=\sigma_{2,3}\cdot\sigma_{1,3}$, therefore
$x\leql y$  does hold.
\end{remark}

\subsection{Garside structure and simple braids}
\label{sec:comb-repr-braids}

\subsubsection{Garside structure}
\label{sec:simple-braids-gars}

The monoid $\bbr n$ is known to be a \emph{Garside
  monoid}~\cite{adian1984fragments,bessis2003dual,birman1998new}; that is to say:
%\begin{inparaenum}[(1)]
\begin{enumerate}[(1)]
\item $\bbr n$~is a cancellative monoid;

\item $\bbr n$~contains a \emph{Garside element}, that is to say, an
    element whose set of left divisors coincides with its set of right
    divisors and contains the generating set~$\Sigma$\,;

\item Every finite subset $X$ of $\bbr n$ has a least upper bound in $(\bbr
    n,\leql)$, and a greatest lower bound if $X\neq\emptyset$, respectively
    denoted $\bigveel X$ and~$\bigwedgel X$.

\end{enumerate}
%\end{inparaenum}

Let $\bigveer X$ denote the least upper bound in $(\bbr n,\leqr)$ of a subset
$X\subseteq\bbr n$\,. Then, if $X$ is a subset of~$\Sigma$, it is
known~\cite{birman1998new,garside1969braid} that $\bigveel X$ and $\bigveer
X$ coincide. We introduce therefore the notation $\Delta_X$ for:
\[
  \Delta_X=\bigveel X=\bigveer X\,,\qquad\text{for $X\subseteq\Sigma$\,.}
\]
Moreover, $\Delta_X$~has the same left divisors and right divisors in~$\bbr
n$: $\{x \in \bbr n \tq x \leql \Delta_X\} = \{x \in \bbr n \tq x \leqr
\Delta_X\}$\,.

The element $\Delta_X$ one obtains when considering $X=\Sigma$ plays a
special role in Garside theory. Indeed, based on the above remarks, it is not
difficult to see that $\Delta_\Sigma$ is a Garside element of~$\bbr n$, and
is moreover the \emph{smallest} Garside element of~$\bbr n$.
Defining the elements $\Delta_n\in\br n$ and $\delta_n\in\dbr n$ by:
\begin{align*}
  \Delta_n&=(\sigma_1\cdot\ldots\cdot\sigma_{n-1})\cdot(\sigma_1\cdot\ldots\cdot\sigma_{n-2})\cdot\ldots\cdot(\sigma_1\cdot\sigma_2)\cdot\sigma_1\,,&
\delta_n&= \sigma_{1,2} \cdot \sigma_{2,3} \cdot \ldots \cdot
             \sigma_{n-1,n}\,,
\end{align*}
we have $\Delta_\Sigma=\Delta_n$ if $\bbr n=\br n$ and
$\Delta_\Sigma=\delta_n$ if $\bbr n=\dbr n$. We adopt the single notation
$\Delta=\Delta_\Sigma$ to denote either $\Delta_n$ or $\delta_n$ according to
which monoid we consider.

\subsubsection{Definition of simple braids}
\label{sec:defin-simple-braids}

The set of all divisors of $\Delta$ is denoted by~$\SS_n$, and its elements
are called \emph{simple braids}.  It is a bounded subset of~$\bbr n$, with
$\unit$ as minimum and $\Delta$ as maximum, closed under $\bigveel$ and
under~$\bigwedgel$\,. With the induced partial order, $(\SS_n,\leql)$ is thus
a lattice.

Consider the mapping $\Phi:\mathcal{P}(\Sigma)\to\SS_n, \ X \mapsto
\Delta_X$, and its image:
\begin{equation*}
  \label{eq:7}
%\text{for }X\in\mathcal P(\Sigma)\quad\Phi(X)&=\Delta_X\,,
\DD_n=\{\Delta_X\tq X\subseteq\Sigma\}\,.
\end{equation*}
Then $\Phi:\bigl(\mathcal P (\Sigma),\subseteq\bigr)\to(\SS_n,\leql)$ is a
lattice homomorphism, and $\DD_n$ is thus a sub-lattice of $(\SS_n,\leql)$.
If $\bbr n=\br n$\,, the mapping $\Phi$ is injective, but \emph{not
onto}~$\SS_n$\,, and so $(\DD_n,\leql)$ is isomorphic to $\bigl(\mathcal
P(\Sigma),\subseteq\bigr)$.  If $\bbr n=\dbr n$, the mapping $\Phi$ is not
injective, but it is onto~$\SS_n$\,, hence $\DD_n=\SS_n$\,.

\begin{remark}
\label{rem:5} Contrasting with the length discussed in Remark~\ref{rem:2},
the notion of simplicity is \emph{not} invariant on~$\br n$\,. For instance,
the braid $\Delta_n$ is simple in~$\br n$\,, but it is not simple as an
element of~$\dbr n$ as soon as $n\geq3$. Indeed, since its length is
$\abs{\Delta_n}=n(n-1)/2$\,, it cannot be a divisor of $\delta_n$ which is of
length $\abs{\delta_n}=n-1$.
\end{remark}

\paragraph*{Running examples for $n=3$.}

The Garside elements of $\br 3$ and of $\dbr 3$ are:
\begin{align*}
\delta_3&=\sigma_{1,2}\cdot\sigma_{2,3}\,,&
  \Delta_3&=\sigma_1\cdot\sigma_2\cdot\sigma_1\,.
\end{align*}

The Hasse diagrams of $(\SS_3,\leql)$ are pictured in
Figure~\ref{fig:hessediagram3}.  For~$\br 3$\,, the lattice $(\DD_3,\leql)$
consists of the following four elements:
\begin{align*}
  \Delta_\emptyset&=\unit&\Delta_{\sigma_1}&=\sigma_1
&\Delta_{\sigma_2}&=\sigma_2&\Delta_{\sigma_1,\sigma_2}&=\sigma_1\cdot\sigma_2\cdot\sigma_1=\Delta
\end{align*}
whereas the lattice $(\SS_3,\leql)$ contains the two additional elements
$\sigma_1\cdot\sigma_2$ and~$\sigma_2\cdot\sigma_1$\,. For~$\dbr 3$\,, the
elements of $\DD_3$ and $\SS_3$ are:
\begin{gather*}
\begin{aligned}
  \Delta_\emptyset &=\unit\qquad
&\Delta_{\{\sigma_{1,2}\}}  &=\sigma_{1,2}\qquad
&\Delta_{\{\sigma_{2,3}\}} &=\sigma_{2,3}\qquad
&\Delta_{\{\sigma_{1,3}\}} &=\sigma_{1,3}
\end{aligned} \\
\Delta_{\{\sigma_{1,2},\sigma_{2,3}\}}=\Delta_{\{\sigma_{2,3},\sigma_{1,3}\}}=
\Delta_{\{\sigma_{1,2},\sigma_{1,3}\}}=\Delta_{\{\sigma_{1,2},\sigma_{2,3},\sigma_{1,3}\}}=\delta_3
\end{gather*}

\begin{figure}
\setbox0=\hbox{$\sigma_1\cdot\sigma_2$}
\newlength{\mylength}\setlength{\mylength}{\wd0}
\newcommand{\myentry}[1]{*+[F]{\makebox[\mylength]{\strut$#1$}}}
  \begin{align*}
% \begin{tikzpicture}
% \node at (0,0) {$\unit$};
% \node at (-1.5,1) {$\sigma_1$};
% \node at (1.5,1) {$\sigma_2$};
% \node at (-1.5,2) {$\sigma_1 \cdot \sigma_2$};
% \node at (1.5,2) {$\sigma_2 \cdot \sigma_1$};
% \node at (0,3) {$\Delta_3$};
% \draw (0,0) circle (0.3);
% \draw (-1.5,1) circle (0.3);
% \draw (1.5,1) circle (0.3);
% \draw (0,3) circle (0.3);
% \draw (-1.5,1.3) -- (-1.5,1.8);
% \draw (1.5,1.3) -- (1.5,1.8);
% \draw (0.25,0.17) -- (1.25,0.83);
% \draw (-0.25,0.17) -- (-1.25,0.83);
% \draw (1,2.2) -- (0.23,2.81);
% \draw (-1,2.2) -- (-0.23,2.81);
% \end{tikzpicture}
% &&
% \begin{tikzpicture}
% \node at (0,0) {$\unit$};
% \node at (-1.5,1) {$\sigma_{1,2}$};
% \node at (0,1) {$\sigma_{2,3}$};
% \node at (1.5,1) {$\sigma_{1,3}$};
% \node at (0,2.5) {$\delta_3$};
% \draw (0,0) circle (0.3);
% \draw (-1.5,1) circle (0.35);
% \draw (1.5,1) circle (0.35);
% \draw (0,1) circle (0.35);
% \draw (0,2.5) circle (0.3);
% \draw (-1.35,1.3) -- (-.3,2.4);
% \draw (1.35,1.3) -- (.3,2.4);
% \draw (0.25,0.17) -- (1.209,0.806);
% \draw (-0.25,0.17) -- (-1.209,0.806);
% \draw (0,1.35) -- (0,2.2);
% \draw (0,.3) -- (0,.65);
% \end{tikzpicture}
\xymatrix@R=1.3em@C=1.5em{
&
\myentry{\Delta_3}\\
\strut \sigma_1\cdot\sigma_2\POS!U\ar@{-}[ur]!D!L
&&\strut\sigma_2\cdot\sigma_1\POS!U\ar@{-}[ul]!D!R\\
\myentry{\sigma_1}\ar@{-}[u]&&\myentry{\sigma_2}\ar@{-}[u]\\
 &\myentry{\unit}\POS!U!L\ar@{-}[ul]!D\POS!R\ar@{-}[ur]!D
}
&&
\xymatrix{
&\myentry{\delta_3}\\
\myentry{\sigma_{1,2}}\POS!U\ar@{-}[ur]!D!L
&\myentry{\sigma_{2,3}}\ar@{-}[u]
&\myentry{\sigma_{1,3}}\POS!U\ar@{-}[ul]!D!R\\
&\myentry{\unit}\POS!U!L\ar@{-}[ul]!D\POS!R\ar@{-}[ur]!D\POS!L(0.5)\ar@{-}[u]
}
  \end{align*}
  \caption{Hasse diagrams of $(\SS_3,\leql)$ for $\br 3$ at left and
    for $\dbr 3$ at right. Elements of $\DD_3$ are framed.}
  \label{fig:hessediagram3}
\end{figure}

\subsubsection{Combinatorial representations of simple braids}
\label{sec:comb-epr-simple}

The natural map that sends each generator $\sigma_i$ of the braid group $\gbr
n$ to the transposition $(i,i+1)$ induces a morphism from $\gbr n$
to~$\mathfrak{S}_n$\,, the set of permutations of $n$ elements.  Hence, this
map also induces morphisms from $\br n$ and from $\dbr n$
to~$\mathfrak{S}_n$.

In the case of the braid monoid $\br n$, this morphism induces a bijection
from $\SS_n$ to~$\mathfrak{S}_n$\,.  Thus, $\SS_n$~has cardinality~$n!$\,. The
element $\unit$~corresponds to the identity permutation, and the element
$\Delta_n$ to the mirror permutation $i\mapsto n+1-i$.

From the point of view of braid diagrams, such as the one pictured in
Figure~\ref{fig:braidsojf}, simple braids correspond to braids such that in
any representative braid diagram, any two strands cross at most once.

%
%In this representation, $\unit$~corresponds to the identity
%permutation, and $\Delta_n$ to the mirror permutation, that is to say,
%as a product of cycles:
%\begin{gather*}
%  \begin{cases}
%    (n,1)\cdot(n-1,2)\cdot\ldots\cdot(n/2+1,n/2),&\text{if $n$ is
%      even}\\
%    (n,1)\cdot(n-1,2)\cdot\ldots\cdot((n+3)/2,(n-1)/2)&\text{if $n$ is
%      odd, leaving $(n+1)/2$ fixed}
%  \end{cases}
%\end{gather*}

\medskip

In the case of the dual braid monoid~$\dbr n$, this morphism only induces an
injection from $\SS_n$ to~$\mathfrak{S}_n$.  It is thus customary to consider
instead the following alternative representation.  Recall that a partition
$\{T^1,\ldots,T^m\}$ of $\{1,\ldots,n\}$ is called \emph{non-crossing} if the
sets $\{\exp(2 \mathbf{i}k \pi / n) \tq k \in T^i\}$ have pairwise disjoint
convex hulls in the complex plane. For instance, on the left of
Figure~\ref{fig:nocrosodoif}, is illustrated the fact that $\{ \{1\},
\{2,3\}, \{4,5,6\} \}$ is a non-crossing partition of $\{1,2,\dots, 6 \}$.

Now, for each subset $T$ of $\{1,\ldots,n\}$, let $x_T$ denote the dual braid
$\sigma_{t_1,t_2} \cdot \sigma_{t_2,t_3} \cdot \ldots \cdot
\sigma_{t_{k-1},t_k}$, where $t_1 < t_2 < \ldots < t_k$ are the elements
of~$T$. Then, for each non-crossing partition $\mathbf{T} =
\{T^1,\ldots,T^m\}$ of $\{1,\ldots,n\}$, we denote by $x_\mathbf{T}$ the
(commutative) product $x_{T^1} \cdot \ldots \cdot x_{T^m}$. It is
known~\cite{birman1998new,bessis02} that the mapping $\mathbf{T} \mapsto
x_{\mathbf{T}}$ is a lattice isomorphism from the set of non-crossing
partitions of $\{1,\ldots,n\}$ onto~$\SS_n$\,. Thus in particular, the
cardinality of $\SS_n$ is the Catalan number~$\frac{1}{n+1} \binom{2n}{n}$. In
this representation, $\unit$~corresponds to the finest partition
$\{\{1\},\ldots,\{n\}\}$, and $\delta_n$ to the coarsest partition
$\{\{1,\ldots,n\}\}$.  See Figure~\ref{fig:nocrosodoif}.

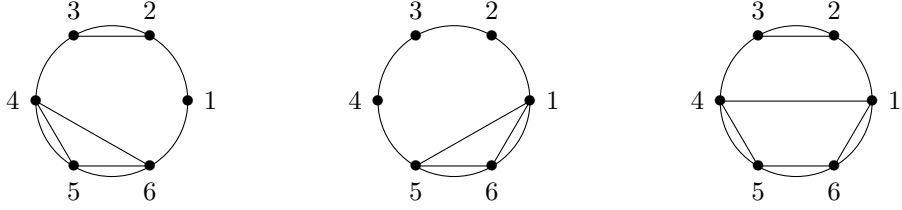
\begin{figure}
\begin{align*}
\begin{tikzpicture}
\node at (1,0) {$\bullet$};
\node at (.5,.86) {$\bullet$};
\node at (-.5,-.86) {$\bullet$};
\node at (-1,0) {$\bullet$};
\node at (-.5,.86) {$\bullet$};
\node at (.5,-.86) {$\bullet$};
\node at (1.3,0) {$1$};
\node at (.5,1.2) {$2$};
\node at (-.5,1.2) {$3$};
\node at (-1.3,0) {$4$};
\node at (-.5,-1.2) {$5$};
\node at (.5,-1.2) {$6$};
\draw (0,0) circle (1);
\draw (.5,.86) -- (-.5,.86);
\draw (-1,0) -- (-.5,-.86);
\draw (-.5,-.86) -- (.5,-.86);
\draw (.5,-.86) -- (-1,0);
\end{tikzpicture}
&&
\begin{tikzpicture}
\node at (1,0) {$\bullet$};
\node at (.5,.86) {$\bullet$};
\node at (-.5,-.86) {$\bullet$};
\node at (-1,0) {$\bullet$};
\node at (-.5,.86) {$\bullet$};
\node at (.5,-.86) {$\bullet$};
\node at (1.3,0) {$1$};
\node at (.5,1.2) {$2$};
\node at (-.5,1.2) {$3$};
\node at (-1.3,0) {$4$};
\node at (-.5,-1.2) {$5$};
\node at (.5,-1.2) {$6$};
\draw (0,0) circle (1);
\draw (-.5,-.86) -- (.5,-.86);
\draw (.5,-.86) -- (1,0);
\draw (1,0) -- (-.5,-.86);
\end{tikzpicture}
&&\begin{tikzpicture}
\node at (1,0) {$\bullet$};
\node at (.5,.86) {$\bullet$};
\node at (-.5,-.86) {$\bullet$};
\node at (-1,0) {$\bullet$};
\node at (-.5,.86) {$\bullet$};
\node at (.5,-.86) {$\bullet$};
\node at (1.3,0) {$1$};
\node at (.5,1.2) {$2$};
\node at (-.5,1.2) {$3$};
\node at (-1.3,0) {$4$};
\node at (-.5,-1.2) {$5$};
\node at (.5,-1.2) {$6$};
\draw (0,0) circle (1);
\draw (.5,.86) -- (-.5,.86);
\draw (-1,0) -- (-.5,-.86);
\draw (-.5,-.86) -- (.5,-.86);
\draw (.5,-.86) -- (1,0);
\draw (1,0) -- (-1,0);
\end{tikzpicture}
\end{align*}
\caption{Left, representation of
  $\sigma_{2,3}\cdot\sigma_{4,5}\cdot\sigma_{5,6}$ in $\dbr 6$ as a
  non-crossing partition of $\{1,\ldots,6\}$. Middle, representation
  of~$\sigma_{1,5}\cdot\sigma_{5,6}$\,. Right, representation of their
  least upper bound
  $(\sigma_{2,3}\cdot\sigma_{4,5}\cdot\sigma_{5,6})\veel(\sigma_{1,5}\cdot\sigma_{5,6})
=\sigma_{2,3}\cdot\sigma_{1,4}\cdot\sigma_{4,5}\cdot\sigma_{5,6}$\,.}
  \label{fig:nocrosodoif}
\end{figure}

\paragraph*{Running examples for $n=3$.}

Let us consider the case  $n=3$. For~$\br 3$\,, the correspondence between
simple braids and permutations of $\{1,2,3\}$ is the following:
\begin{align*}
  \unit&=\text{Id}&\sigma_1&=(1,2)&\sigma_2&=(2,3)\\
\sigma_1\cdot\sigma_2&=(1,2,3)&\sigma_2\cdot\sigma_1&=(1,3,2)&\Delta&=(3,1)
\end{align*}

Simple braids of $\dbr 3$ correspond to non-crossing partitions of
$\{1,2,3\}$, which in this case are simply all the partitions of $\{1,2,3\}$.
The correspondence is the following, where singletons are omitted:
\begin{align*}
  \unit&=\bigl\{\bigr\}&
\sigma_{1,2}&=\bigl\{\{1,2\}\bigr\}&
\sigma_{2,3}&=\bigl\{\{2,3\}\bigr\}&
\sigma_{1,3}&=\bigl\{\{1,3\}\bigr\}&
\delta_3&=\bigl\{\{1,2,3\}\bigr\}
\end{align*}

\section{Garside normal form and combinatorics of braids}
\label{sec:garside-normal-form}

Braids are known to admit normal forms, that is to say, a unique
combinatorial representation for each braid. Normal forms are the standard
tool for several algorithmic problems related to braids, for instance the
word problem to cite one of the most emblematic~\cite{dehornor08}. Among the
several normal forms for braids introduced in the literature, we shall focus
in this work on the Garside normal form which derives from the Garside
structure attached to our braid monoids, as recalled above.

\subsection{Garside normal form of braids}
\label{sec:combinatorics-braids}

In the monoid $\bbr n$, and regardless on whether $\bbr n = \br n$ or $\bbr n
= \dbr n$, a sequence $(x_1,\ldots,x_k)$ of \emph{simple
  braids} is said to be \emph{normal} if
$x_j=\Delta \wedgel (x_j\cdot\ldots\cdot x_k)$ holds for all $j=1,\ldots,k$.
Intuitively, this is a maximality property, meaning that no left divisor of
$x_{j+1} \dotsm x_k$ could be moved to $x_j$ while remaining in the world of
simple braids. We recall the two following well known facts concerning normal
sequences of braids:
\begin{enumerate}
\item\label{item:2} For $x,y\in\SS_n$, let $x\to y$ denote the relation
    $\Dr(x)\supseteq\Dl(y)$, where the sets $\Dr(x)$ and $\Dl(y)$ are
    defined as follows:
\begin{align}
\label{eq:9}
\Dr(x)& = \bigl\{ \sigma \in \Sigma \tq x\cdot \sigma \notin
        \SS_n\}\,,
&\Dl(y) &=\bigl\{\sigma \in \Sigma \tq \sigma\leql y\bigr\}\,.
\end{align}
Then a sequence $(x_1,\ldots,x_k)$ is normal if and only if $x_j\to
x_{j+1}$ holds for all $j=1,\ldots,k-1$, again meaning that left divisors
of $x_{j+1}$ are already present in $x_j$, and therefore cannot be pushed
into $x_j$ while keeping it simple.
\item For every non-unit braid $x\in\bbr n$, there exists a unique normal
    sequence $(x_1,\ldots,x_k)$ of non-unit simple braids such that
    $x=x_1\cdot\ldots\cdot x_k$. This sequence is called the \emph{Garside
    normal form} or \emph{decomposition} of~$x$.
\end{enumerate}

In this work, the integer $k$ is called the \emph{height} of~$x$ (it is also
called the \emph{supremum} of $x$ in the literature~\cite{ECHLPT}). We denote
it by~$\height(x)$.

Regarding the special elements $\unit$ and~$\Delta$, the following dual
relations hold, meaning that $\unit$ is ``final'' whereas $\Delta$ is
``initial'':
\begin{align*}
\forall x\in\SS_n\quad x &\to \unit & \forall x\in\SS_n\quad
  \unit\to x &\iff x=\unit \\
\forall x\in\SS_n\quad \Delta &\to x & \forall x\in\SS_n\quad  x \to \Delta &\iff x=\Delta
\end{align*}

Therefore, the Garside decomposition starts with a finite (possibly zero)
number of~$\Delta$s, and then is given by a finite path in the finite
directed graph $(\SS_n \setminus \{\Delta, \unit\},\to)$. By convention, we
define the Garside normal form of the unit braid $\unit$ as the sequence
$(\unit)$, and we put $\height(\unit)=1$. (It might seem that
$\height(\unit)=0$ would be a more natural convention, but it turns out that
taking $\height(\unit)=1$ is the good choice for convenient formulation of
several results below as it encompasses the fact that, in normal forms,
$\unit$ can not be followed by any letter. For instance,
Lemma~\ref{lem:full-visual-is-union-of-garside} would not hold otherwise.)
Then, it is a well-known property of Garside monoids that $\height(x)$ is the
least positive integer $k$ such that $x$ is a product of $k$ simple braids.

Moreover, it will be convenient to complete the normal form of a braid with
infinitely many factors~$\unit$. We call the infinite sequence
$(x_k)_{k\geq1}$ of simple braids thus obtained the \emph{extended
  Garside decomposition} of the braid. The directed graph
$(\SS_n,\to)$ is then their accepting graph: extended Garside decompositions
of braids correspond bijectively to infinite paths in~$(\SS_n,\to)$ that
eventually hit~$\unit$, and then necessarily stay in $\unit$ forever.

\begin{remark}
%\textbf{Explications supplémentaires sur le fait que les longueurs coïncident~?}\par
\label{rem:4} Following up on Remark~\ref{rem:5}, just as simplicity was
observed not to be invariant on~$\br n$, the height and the Garside normal
forms are not invariant on~$\br n$\,.  For instance, the braid $\Delta_n$ has
Garside normal form the sequence $(\Delta_n)$ itself in~$\br n$, and the
sequence $(\delta_n,\delta_{n-1},\ldots,\delta_2)$ in~$\dbr n$.  Its height
is $1$ in $\br n$ and $n-1$ in~$\dbr n$.
  %
  % Consequently, in what follows, we shall really see braids as
  % elements of a given monoid~$\bbr n$ (which may be either $\br n$ or
  % $\dbr n$) and the associated notions of length, of simplicity, of
  % Garside normal form and of height are those associated with the
  % monoid~$\bbr n$.
\end{remark}

We gather in the following proposition some well-known properties of Garside
normal forms~\cite{dehornoy2013foundations} that we shall use later.
\begin{proposition}
\label{proposition:dehornoy2013foundations} For all braids $x,y \in \bbr n$:
\begin{enumerate}
\item\label{pro:garside-1} the height $\height(x)$ is the smallest integer
    $k\geq 1$ such that $x\leql \Delta^k$;
\item\label{pro:garside-2} if $(x_1,\ldots,x_k)$ is the normal form of $x$,
    then $x_1\cdot\ldots\cdot x_j=x\wedgel \Delta^j$ for all
    $j\in\{1,\ldots,k\}$ ;
\item\label{pro:garside-3} $x\leql y\iff x\leq y\wedgel
    \Delta^{\height(x)}$.
\end{enumerate}
\end{proposition}

\begin{remark}
  \label{rem:1}
We should stress that, while the normal form is very convenient to enumerate
braids, it behaves poorly with respect to multiplication: consider $x$ with
height $k$ and Garside decomposition $(x_1,\dotsc,x_k)$, and let $\sigma$ be
a generator. Then, $y=x\cdot \sigma$ has height in $\{k, k+1\}$, but if it
has height $k$ then the normal form $y=(y_1,\dotsc, y_k)$ might be completely
different from that of $x$ (in the sense that $y_1\neq x_1,\dotsc, y_k \neq
x_k$), although it is algorithmically computable.
\end{remark}

\paragraph*{Running examples for $n=3$.}

Let us describe explicitly the accepting graphs $(\SS_3,\to)$ for $\br 3$ and
for~$\dbr 3$. Consider first the case of~$\br 3$\,. The subsets $\Dl(x)$ and
$\Dr(x)$ are easily computed through their definition~(\ref{eq:9}), from
which the relation $\to$ is derived. The results of these computations are
depicted in Figure~\ref{fig:acceptorB3}. The analogous computations for $\dbr
3$ result in the data pictured in Figure~\ref{fig:acceptordual}.

\begin{figure}
  \begin{align*}
    \begin{array}[t]{rcll}
      \Dl(x)&x\in\SS_3&\Dr(x)&\{y\in\SS_3\tq x\to y\}\\
\hline
\emptyset&\unit&\emptyset&\{\unit\}\\
\{\sigma_1\}&\sigma_1&\{\sigma_1\}&\{\unit,\sigma_1,\sigma_1\cdot\sigma_2\}\\
\{\sigma_2\}&\sigma_2&\{\sigma_2\}&\{\unit,\sigma_2,\sigma_2\cdot\sigma_1\}\\
\{\sigma_1\}&\sigma_1\cdot\sigma_2&\{\sigma_2\}&\{\unit,\sigma_2,\sigma_2\cdot\sigma_1\}\\
\{\sigma_2\}&\sigma_2\cdot\sigma_1&\{\sigma_1\}&\{\unit,\sigma_1,\sigma_1\cdot\sigma_2\}\\
\Sigma&\Delta_3&\Sigma&\SS_3
    \end{array}&&
\xymatrix{
&*+[F]{\strut\sigma_1}\POS!U!L\ar@(ul,ur)!U!R
\POS[]\ar[r]
\POS[]\POS!D!R\ar[drr]!U!L
&*+[F]{\strut\sigma_1\cdot\sigma_2}
\ar@{<->}[dd]
\POS!R\ar[dr]!U
\POS[]\POS!D!L(.5)\ar[ddl]!U
\\
*+[F]{\strut\Delta}
\POS!U!L\ar@(ul,dl)!D!L
\POS!R\POS!U\ar[ur]!D!L
\POS[]\POS!R!U(.5)\ar[urr]!D!L
\POS[]\ar[rrr]
\POS[]\POS!R\POS!D(.5)\ar[drr]!U!L
\POS[]\POS!R\POS!D\ar[dr]!U!L
&&&*+[F]{\strut\unit}
\POS!U!R\ar@(ur,dr)!D!R
\\
&*+[F]{\strut\sigma_2}
\POS!D!L\ar@(dl,dr)!D!R
\POS[]\ar[r]
\POS[]\POS!R!U\ar[urr]!D!L
&*+[F]{\strut\sigma_2\cdot\sigma_1}
\POS!R\ar[ur]!D
\POS[]\POS!U!L(.5)\ar[uul]!D
}\\
  \end{align*}
  \caption{Left hand side: subsets $\Dl(x)$ and $\Dr(x)$ for simple
    braids $x\in\br 3$. Right hand side: the resulting accepting
    graph $(\SS_3,\to)$ for $\br 3$}
  \label{fig:acceptorB3}
\end{figure}

\begin{figure}
  \begin{align*}
    \begin{array}[t]{rcll}
      \Dl(x)&x\in\SS_3&\Dr(x)&\{y\in\SS_3\tq x\to y\}\\
\hline
\emptyset&\unit&\emptyset&\{\unit\}\\
\{\sigma_{1,2}\}&\sigma_{1,2}&\{\sigma_{1,2},\;\sigma_{1,3}\}&\{\unit,\sigma_{1,2},\;\sigma_{1,3}\}\\
\{\sigma_{2,3}\}&\sigma_{2,3}&\{\sigma_{2,3},\; \sigma_{1,2}\}&\{\unit,\sigma_{2,3},\; \sigma_{1,2}\}\\
\{\sigma_{1,3}\}&\sigma_{1,3}&\{\sigma_{1,3},\;\sigma_{2,3}\}&\{\unit,\sigma_{1,3},\;\sigma_{2,3}\}\\
\Sigma&\delta_3&\Sigma&\SS_3
    \end{array}  &&
\xymatrix{
*+[F]{\strut\unit}\POS!U!L\ar@(ul,ur)!U!R
&&*+[F]{\strut\sigma_{2,3}}\POS!D!R\ar@(dr,ur)!U!R\POS!L!U(.25)\ar[dll]!U!R\POS[]\ar[ll]\\
*+[F]{\strut\sigma_{1,2}}\POS!D!L\ar@(dl,ul)!U!L\POS[]\ar[rr]\ar[u]&&
*+[F]{\strut\sigma_{1,3}}\POS!D!R\ar@(dr,ur)!U!R\POS[]\ar[u]\POS!U!L\ar[ull]\\
&*+[F]{\strut\delta_3}\POS!D!L\ar@(dl,dr)!D!R\POS[]!U\ar[uul]\POS!U!L\ar[ul]!D!R
\POS!R(.75)\ar[uur]!D!L\POS!R\ar[ur]!D!L
}\\
  \end{align*}
  \caption{Left hand side: subsets $\Dl(x)$ and $\Dr(x)$ for simple
    braids $x\in\dbr 3$. Right hand side: the resulting accepting
    graph $(\SS_3,\to)$ for $\dbr 3$}
  \label{fig:acceptordual}
\end{figure}

\subsection{Combinatorics of braids}
\label{sec:combinatorics-braids-1}

How many braids $x\in\bbr n$ of length $k$ are there? Is there either an
exact or an approximate formula? The aim of this subsection is to recall the
classical answers to these questions, which we will do by analyzing the
ordinary generating function, or growth series, of~$\bbr n$\,. This series
turns out to be rational, and the Garside normal form is an efficient tool
for the study of its coefficients.

\subsubsection{Growth series of braid monoids and Möbius polynomial}
\label{sec:growth-series-braid}

Let $G_n(t)$ be the growth series of~$\bbr n$. It is defined by:
\[
  G_n(t)=\sum_{x\in\bbr n}t^{\abs{x}}\,.
\]

According to a well-known result~\cite{deligne72,charney95,bronfman01}, the
growth function of $\bbr n$ is rational, inverse of a polynomial:
\begin{align}
\label{eq:2}
  G_n(t)&=\frac1{H_n(t)}\,,&
\text{with\quad}  H_n(t)&=\sum_{X\subseteq \Sigma}(-1)^{\abs{X}}\,t^{\abs{\Delta_X}}\,.
\end{align}

There exist explicit or recursive formul\ae\ allowing to compute
effectively~$H_n(t)$:
\begin{align*}
  H_n(t)&=\sum_{k=1}^n(-1)^{k+1}t^{\frac{k(k-1)}{2}}H_{n-k}(t) &&
                                                                 \text{ if } \bbr n = \br n\\
  H_n(t)&= \sum_{k=0}^{n-1} (-1)^k \frac{(n-1+k)!}{(n-1-k)!k!(k+1)!} t^k && \text{ if } \bbr n = \dbr n
\end{align*}
% \[\def\arraystretch{2.8}\begin{array}{ll}
%                           \displaystyle H_n(t)=\sum_{k=1}^n(-1)^{k+1}t^{\frac{k(k-1)}{2}}H_{n-k}(t) & \text{ if } \bbr n = \br n \\
%                           \displaystyle H_n(t)= \sum_{k=0}^{n-1} (-1)^k \frac{(n-1+k)!}{(n-1-k)!k!(k+1)!} t^k & \text{ if } \bbr n = \dbr n
% \end{array}\]

For reasons that will appear more clearly in a moment (see
Subsection~\ref{sec:two-mobi-transf}), the polynomial $H_n(t)$ deserves the
name of \emph{Möbius polynomial of $\bbr n$}.

\paragraph*{Running examples for $n=3$.}

For $\br 3$, the computation of the Möbius polynomial may be done as follows:
\begin{align*}
  H_3(t)&=1-t^{|\sigma_1|}-t^{|\sigma_2|}+t^{|\sigma_1\vee\sigma_2|}=1-2t+t^3\qquad\text{since
                                                                               $\sigma_1\vee\sigma_2=\Delta_3$}\\
\intertext{and similarly for $\dbr 3$:}
H_3(t)&=1-t^{|\sigma_{1,2}|}-t^{|\sigma_{2,3}|}-t^{|\sigma_{1,3}|}+t^{|\sigma_{1,2}\vee\sigma_{1,3}|}
+t^{|\sigma_{1,2}\vee\sigma_{2,3}|}+t^{|\sigma_{2,3}\vee\sigma_{1,3}|}-t^{|\delta_3|}\\
&=1-3t+2t^2\quad\qquad\text{since
  $\sigma_{1,2}\vee\sigma_{1,3}=\sigma_{1,2}\vee\sigma_{2,3}=\sigma_{2,3}\vee\sigma_{1,3}=\delta_3$}
\end{align*}

\subsubsection{Connectivity of the Charney graph}
\label{sec:charney-graph}

The growth series $G_n(t)$ is a rational series with non-negative
coefficients and with a finite positive radius of convergence, say~$q_n$\,,
which, by the Pringsheim theorem~\cite{flajolet09}, is necessarily itself a pole
of~$G_n(t)$. Since $G_n(t)=1/H_n(t)$ as recalled in~\eqref{eq:2}, it follows
that $q_n$ is a root of minimal modulus of the polynomial~$H_n(t)$.  In order
to evaluate the coefficients of $G_n(t)$, we shall prove that $G_n(t)$ has no
other pole of modulus~$q_n$\,, or equivalently, that $H_n(t)$ has no other
root of modulus~$q_n$\,. This is stated in Corollary~\ref{cor:6} below.

To this end, we first study the connectivity of the \emph{Charney
  graph}, which is the directed graph $\mathcal{G} = (V,E)$ with set
of vertices $V = \SS_n \setminus \{\unit,\Delta\}$ and set of edges $E =
\{(x,y) \in V^2 \tq x \to y\}$. The connectivity of $\mathcal{G}$ is well
known for $\bbr n=\br
n$~\cite{bestvina1999non,caruso2013genericity,gebhardt2014penetration}, and
actually the same result also holds for $\bbr n=\dbr n$\,, although it does
not seem to be found in the literature. We obtain thus the following result.

\begin{proposition}
\label{prop:3} For $n\geq3$, the Charney graph of $\bbr n$ is strongly
connected and contains loops.
\end{proposition}

\begin{proof}
First, observe that the graph $\mathcal{G}$ contains the loop $\sigma \to
\sigma$ for every generator $\sigma \in \Sigma$.  Since proofs of the strong
connectivity of $\mathcal{G}$ are found in the literature when $\bbr n = \br
n$, we focus on proving that $\mathcal{G}$ is strongly connected when $\bbr n
= \dbr n$.

Recall that simple braids are in bijection with non-crossing partitions of
$\{1,\ldots,n\}$.  For each subset $T$ of $\{1,\ldots,n\}$, we denote by
$x_T$ the braid $\sigma_{t_1,t_2} \cdot \sigma_{t_2,t_3} \cdot \ldots \cdot
\sigma_{t_{k-1},t_k}$, where $t_1 < t_2 < \ldots < t_k$ are the elements
of~$T$.  Then, for each non-crossing partition $\mathbf{T} =
\{T^1,\ldots,T^m\}$ of $\{1,\ldots,n\}$, we denote by $x_\mathbf{T}$ the
(commutative) product $x_{T^1} \cdot \ldots \cdot x_{T^m}$.  It is
known~\cite{birman1998new} that
\begin{inparaenum}[(1)]
\item the mapping $\mathbf{T} \mapsto x_{\mathbf{T}}$ is a bijection
  from the set of non-crossing partitions of $\{1,\ldots,n\}$ to
  $\SS_n$, as mentioned in Subsection~\ref{sec:comb-repr-braids};
\item the set $\Dl(x_{\mathbf{T}})$ is equal to
  $\{\sigma_{u,v} \tq \exists T \in \mathbf{T}\quad u,v \in
  T\}$;
\item the set $\Dr(x_{\mathbf{T}})$ is equal to
  \[
\bigl\{\sigma_{u,v} \tq \exists T \in \mathbf{T} \quad
  T \cap \{u+1,\ldots,v\} \neq \emptyset \text{ and } T \cap \{1,\ldots,u,v+1,\ldots,n\} \neq \emptyset\bigr\}\,.
\]
\end{inparaenum}

Hence, consider two braids $y,z \in \SS_n \setminus\{\unit,\Delta\}$, and let
$m = \lfloor \frac{n}{2} \rfloor$, as well as the set $A =
\{\sigma_{1,2},\ldots,\sigma_{n-1,n},\sigma_{n,1}\}$.  Since $z <_{\text{l}}
\Delta$, we know that $z = x_{\mathbf{Z}}$ where $\mathbf{Z}$ is a partition
of $\{1,\ldots,n\}$ in at least two subsets.  It follows that $\mathbf{Z}$ is
a refinement of a non-crossing partition $\mathbf{V}$ of $\{1,\ldots,n\}$ in
\emph{exactly} two subsets.  The map $\sigma_{i,j} \mapsto \sigma_{i+1,j+1}$
induces an automorphism of the dual braid monoid~$\dbr n$, hence we assume
without loss of generality that $\mathbf{V} =
\bigl\{\{1,\ldots,i,n\},\{i+1,\ldots,n-1\}\bigr\}$ for some integer $i \in
\{0,\ldots,m-1\}$.

Finally, consider some generator $\sigma_{a,b} \in \Dl(y)$, with
$a < b$.  Since $\Dl(y) \subseteq \Dr(y)$, it comes that that
$y \to \sigma_{a,b} \to \sigma_{1,a} \to \sigma_{1,n}\to \sigma_{2,n}$
if $1 < a$, or $y \to \sigma_{1,b} \to \sigma_{1,n} \to \sigma_{2,n}$
if $a = 1$. Since $2 \leq m+1 < n$, we then observe that
$\sigma_{2,n} \to \sigma_{1,m+1} \to x_{\mathbf{T}} \to x_{\mathbf{U}}
\to x_{\mathbf{V}} \to x_{\mathbf{Z}} = z$, where
\begin{align*}
 \mathbf{T} &=
              \begin{cases}
                \bigl\{\{u,n+1-u\} \tq 1 \leq u \leq m\bigr\} & \text{if $n$ is even,} \\
                \bigl\{\{u,n+1-u\} \tq 1 \leq u \leq m\bigr\} \cup
                \bigl\{\{m+1\}\bigr\} & \text{if $n$ is odd;}
\end{cases}\\
 \mathbf{U} &=  \bigl\{\{1,\ldots,m\},\{m+1,\ldots,n\}\bigr\}.
\end{align*}
% \[\arraycolsep=0.3pt\def\arraystretch{1.2}
% \begin{array}{lll}
% \mathbf{T} = & \{\{u,n+1-u\} \tq 1 \leq u \leq m\} & \text{ if } n \text{ is even, or} \\
% & \{\{u,n+1-u\} \tq 1 \leq u \leq m\} \cup \{\{m+1\}\} & \text{ if } n \text{ is odd;} \\
% \mathbf{U} = & \{\{1,\ldots,m\},\{m+1,\ldots,n\}\}. & \\
% \end{array}\]
This completes the proof.
\end{proof}

The connectivity of the Charney graph stated in the above proposition has
several consequences on the combinatorics of braids, which we gather in the
following corollary, the result of which seems to have been unnoticed so far.

\begin{corollary}
  \label{cor:6}
  The Möbius polynomial $H_n(t)$ has a unique root of smallest
  modulus. This root, say~$q_n$\,, is real and lies in $(0,1)$ and it
  is simple. It coincides with the radius of convergence of the growth
  series~$G_n(t)$.

  Furthermore, for each integer $k\geq0$, put
  $\lambda_n(k)=\# \{x\in\bbr n\tq\abs{x}=k\}\,.$ Then, for $n\geq3$,
  the following asymptotics hold for some constant $C_n>0$:
  \begin{gather}
    \label{eq:12}
    \lambda_n(k)\sim_{k\to\infty} C_n q_n^{-k}\,.
  \end{gather}
\end{corollary}

\begin{proof}
  Recall that we have already defined~$q_n$\,, at the beginning of
  Subsection~\ref{sec:charney-graph}, as the radius of convergence of
  $G_n(t)$, and we know that $q_n$ is a root of smallest modulus
  of~$H_n(t)$. We will now derive the two statements of the corollary
  through an application of Perron-Frobenius theory for primitive
  matrices (see, \emph{e.g.},~\cite{seneta81}).

  Let $I_\Delta$ and $I_{\neg\Delta}$ denote the sets
\begin{align*}
I_\Delta&=\{(i,\Delta) \tq 1 \leq i \leq \abs{\Delta}\}\,,&
I_{\neg \Delta}&=\{(i,\sigma) \tq \sigma \in \SS_n \setminus \{\unit,\Delta\} \text{
    and } 1 \leq i \leq \abs{\sigma}\}\,,
\end{align*}
and let $I$ denote the disjoint union $I_\Delta \cup I_{\neg \Delta}$. Let $M
= (M_{x,y})_{x,y \in I}$ be the non-negative matrix defined as follows:
  \[
    M_{(i,\sigma),(j,\tau)}=
    \begin{cases}
1,&\text{if $j=i+1$ and $\sigma=\tau$,}\\
1,&\text{if $i=|\sigma|$ and $j=1$ and $\sigma\to\tau$,}\\
0,&\text{otherwise.}
    \end{cases}
  \]

  By construction, $M$ is a block triangular matrix
  $M=\left(\begin{smallmatrix}A&B\\0&C
    \end{smallmatrix}
    \right)$
%\[\left(\begin{array}{cc}A & B \\ 0 & C\end{array}\right),\]
where $A$, $B$ and $C$ are the restrictions of $M$ to the respective sets of
indices $ I_\Delta \times I_\Delta$, $I_\Delta \times I_{\neg \Delta}$ and
$I_{\neg \Delta} \times I_{\neg \Delta}$.

Since the Charney graph is strongly connected and contains loops according to
Proposition~\ref{prop:3}, and since it contains at least $n-1$ vertices (the
elements of~$\Sigma$), it follows that $C$ is a primitive matrix with Perron
eigenvalue $\rho > 1$.  By construction, we know that $A^{\abs{\Delta}} =
\mathbf{Id}_{\abs{\Delta}}$, hence that the eigenvalues of $A$ have
modulus~$1$.  Consequently, $\rho$~is a simple eigenvalue of~$M$, and has a
strictly greater modulus than all other eigenvalues of~$M$.  Hence, there
exist left and right eigenvectors $\mathbf{l}$ and $\mathbf{r}$ of $M$ for
the eigenvalue $\rho$ with non-negative entries, whose restrictions
$(\mathbf{l}_x)_{x \in I_{\neg \Delta}}$ and $(\mathbf{r}_x)_{x \in I_{\neg
\Delta}}$ only have positive entries, and such that $\mathbf{l} \cdot
\mathbf{r} = 1$.

Then, observe that $\lambda_n(k) = \mathbf{u} \cdot M^{k-1} \cdot \mathbf{v}$
for all $k \geq 1$, where $\mathbf{u}$ is the row vector defined by
$\mathbf{u}_{(i,\sigma)} = \un{i = 1}$ and $\mathbf{v}$ is the column vector
defined by $\mathbf{v}_{(i,\sigma)} = \un{i = \abs{\sigma}}$. Indeed, this
follows at once from the existence and uniqueness of the Garside normal form
for braids, and from the construction of the matrix~$M$.

Both vectors $\mathbf{u}$ and $\mathbf{v}$ have some non-zero entries
in~$I_{\neg \Delta}$\,, and therefore
\begin{gather}
\label{eq:8}
\lambda_n(k) = \mathbf{u} \cdot M^{k-1} \cdot \mathbf{v} \sim_{k\to\infty}
 \rho^{k-1} (\mathbf{u} \cdot \mathbf{r}) (\mathbf{l} \cdot
 \mathbf{v})
\end{gather}
Hence, $\rho^{-1}$~is the radius of convergence of the generating series
$G_n(t) = \sum_{k \geq 0} \lambda_n(k)$, and thus $\rho^{-1}=q_n$\,.

To complete the proof, consider the decomposition of $H_n(t)$ as a product of
the form:
\begin{gather*}
  H_n(t)=(1-t/q_n)\cdot(1-t/r_1)\cdot\ldots\cdot(1-t/r_i)\cdot(1-t/a_1)\cdot\ldots\cdot(1-t/a_j)\,,
\end{gather*}
where $r_1,\ldots,r_i$ are the other complex roots of $H_n(t)$ of
modulus~$q_n$\,, including~$q_n$ if its multiplicity is $>1$, and
$a_1,\ldots,a_j$ are the remaining complex roots of~$H_n(t)$, hence of
modulus greater than~$q_n$\,. Since we know that $G_n(t)=1/H_n(t)$, and
considering the series expansion of $1/H_n(t)$, one sees that the equivalent
found in~(\ref{eq:8}) for the coefficients $\lambda_n(\cdot)$ of $G_n(t)$
cannot hold if $i>0$, whence the result.
\end{proof}

\subsection{Two Möbius transforms}
\label{sec:two-mobi-transf}

This last subsection is devoted to the study of Möbius transforms.
In~\S~\ref{sec:mobius-transform}, we particularize to our braid monoids $\bbr
n$ the classical Möbius transform, as defined for general classes of partial
orders~\cite{rota64,stanley97}. We prove the Möbius inversion formula for our
particular case, although it could be derived from more general results.
Next, we introduce in~\S~\ref{sec:grade-mobi-transf} a variant, called graded
Möbius transform, which will prove to be most useful later for the
probabilistic analysis.

We will use extensively the notation $\un A$ for the characteristic function
of~$A$, equal to $1$ if $A$ is true and to $0$ is $A$ is false.

\subsubsection{The standard Möbius transform}
\label{sec:mobius-transform}

In the framework of braid monoids, the usual Möbius transform is defined as
follows and leads to the next proposition.

\begin{definition}
  \label{def:3}
  Given a real-valued function $f: \bbr n\mapsto\bbR$, its
  \emph{Möbius transform} is the function $h:\bbr n\mapsto\bbR$
  defined by:
\begin{gather}
  \label{eq:4}
h(x)=\sum_{X \subseteq \Sigma} (-1)^{|X|} f(x \cdot \Delta_X) \text{ for all } x \in \bbr n.
\end{gather}
\end{definition}

\begin{proposition}
  \label{prop:2}
Let $f,h:\bbr n\to\bbR$ be two functions such that the series $\sum_{x \in
\bbr n} |f(x)|$ and $\sum_{x \in \bbr n} |h(x)|$ are convergent. Then $h$ is
the Möbius transform of $f$ if and only if
\begin{gather}
\label{eq:10}
\forall x\in\bbr n\quad f(x)=\sum_{y\in\bbr n} h(x \cdot y)\,.
\end{gather}
\end{proposition}

\begin{proof}
For every non-unit braid $z \in \bbr n$, the sets $\Dl(z)$ and $\Dl(z^*)$ are
non-empty, hence the powersets $\mathcal{P} \Dl(z)$ and $\mathcal{P}\Dl(z^*)$
are non-trivial Boolean lattices.  It follows from the equality $\Delta_X =
\bigveel X$ that:
\[\sum_{X \subseteq \Sigma} (-1)^{|X|} \un{\Delta_X \leql z} = \sum_{X \subseteq \Dl(z)} (-1)^{|X|} = 0\,.
\]
And, similarly, it follows from the equality $\Delta_X = \bigveer X$ that:
\[\sum_{X \subseteq \Sigma} (-1)^{|X|} \un{\Delta_X \leqr z} = \sum_{X \subseteq \Dl(z^*)} (-1)^{|X|} = 0\,
.\]

Consider now two functions $f$ and $h$ such that the series $\sum_{x \in \bbr
n} |f(x)|$ and $\sum_{x \in \bbr n} |h(x)|$ are convergent. Assume first that
$h$ is the Möbius transform of $f$. Using the change of variable $v = y \cdot
\Delta_X$, we derive from~\eqref{eq:4} the following:
\begin{align*}
f(x) &= \sum_{v \in \bbr n} \Bigl(\sum_{X \subseteq \Sigma} (-1)^{|X|} \un{\Delta_X \leqr v}\Bigr) f(x \cdot v) \\
&=\sum_{y \in \bbr n} \sum_{X \subseteq \Sigma} (-1)^{|X|} f(x \cdot y \cdot \Delta_X)
= \sum_{y\in\bbr n} h(x \cdot y),
\end{align*}
proving~\eqref{eq:10}.

Conversely, if~\eqref{eq:10} holds, we use the change of variable $u =
\Delta_X \cdot y$ to derive:
\begin{align*}
h(x) &= \sum_{u \in \bbr n} \Bigl(\sum_{X \subseteq \Sigma} (-1)^{|X|} \un{\Delta_X \leql u}\Bigr) h(x \cdot u) \\
&=\sum_{y \in \bbr n} \sum_{X \subseteq \Sigma} (-1)^{|X|} h(x \cdot \Delta_X \cdot y)
= \sum_{X \subseteq \Sigma} (-1)^{|X|} f(x \cdot\Delta_X).
\end{align*}
This shows that $h$ is the Möbius transform of $f$, completing the proof.
\end{proof}

In particular, observe that, if a function $f$ has support in~$\SS_n$\,, then
so does its Möbius transform~$h$.  Hence, we also define the notion of Möbius
transform of real-valued functions $f : \SS_n \to \bbR$ in a natural way.  In
that narrower context, Proposition~\ref{prop:2} formulates as follows.

\begin{corollary}
  \label{cor:5}
Let $f,h:\SS_n\to\bbR$ be two functions. Then the two statements:
\begin{align}
\label{eq:1}
\forall x\in\SS_n\quad  f(x)&=\sum_{y\in\SS_n} \un{x \cdot y \in
                              \SS_n} h(x \cdot y) \\
\label{eq:3}
\forall x\in\SS_n\quad  h(x)&=\sum_{X \subseteq \Sigma} (-1)^{|X|} \un{x \cdot \Delta_X \in
    \SS_n} f(x \cdot \Delta_X)
\end{align}
are equivalent.
\end{corollary}

In particular, by comparing the expressions~(\ref{eq:2}) of $H_n$
and~(\ref{eq:3}) of the Möbius transform of $f:\SS_n\to\bbR$, we observe that
if $p$ is a real number, and if $f:\SS_n \to \bbR$ is defined by
$f(x)=p^{\abs{x}}$, then its Möbius transform $h$ satisfies:
\begin{gather}
  \label{eq:19}
h(\unit)=H_n(p).
\end{gather}

\paragraph*{Running examples for $n=3$.}

We tabulate in Table~\ref{tab:mobiustransform3} the values of the Möbius
transform of the function $p^{|x|}$ defined on~$\SS_3$\,, for $\br 3$ and
for~$\dbr 3$\,. It is easily computed based on the elements found in
Figures~\ref{fig:acceptorB3} and~\ref{fig:acceptordual} respectively.

\begin{table}
\begin{align*}
    \begin{array}{ll}
   x\in\SS_3  &h(x)\\
\hline
\unit&1-2p+p^3=H_3(p)\\
\sigma_1&p-p^2\\
\sigma_2&p-p^2\\
\sigma_1\cdot\sigma_2&p^2-p^3\\
\sigma_2\cdot\sigma_1&p^2-p^3\\
\Delta_3&p^3
    \end{array}
&&
   \begin{array}{ll}
   x\in\SS_3  &h(x)\\
\hline
\unit&1-3p+2p^2=H_3(p)\\
\sigma_{1,2}&p-p^2\\
\sigma_{2,3}&p-p^2\\
\sigma_{1,3}&p-p^2\\
\delta_3&p^2\\
\phantom{\Delta_3}
    \end{array}
\end{align*}
\caption{Values of the Möbius transform $h:\SS_3\to\bbR$ of the
    function $f:\SS_3\to\bbR$ defined by $f(x)=p^{|x|}$
    for~$\br 3$ (left hand side) and for~$\dbr 3$ (right hand side)}
  \label{tab:mobiustransform3}
\end{table}

\subsubsection{The graded Möbius transform}
\label{sec:grade-mobi-transf}

The above relation between real-valued functions $f : \bbr n \mapsto \bbR$
and their Möbius transforms works only when the Möbius transform is summable.
In order to deal with all functions defined on $\bbr n$, we introduce a
variant of those transforms, which is the notion of \emph{graded Möbius
transform}. To this end, for each finite braid $x\in\bbr n$, we define $\bbr
n[x]$ as the following subset:
\[
\bbr n[x] =  \{y\in\bbr n\tq
\height(x \cdot y)=\height(x)\}
\,,
\]

\begin{definition}
\label{def:4} Given a real-valued function $f: \bbr n\mapsto\bbR$, its
\emph{graded Möbius transform} is the function $h:\bbr n\mapsto\bbR$ defined
by:
\begin{gather}
  \label{eq:4**}
\forall x\in\bbr n\quad  h(x)=\sum_{X \subseteq \Sigma} (-1)^{|X|} \un{\Delta_X \in \bbr
    n[x]} f(x \cdot \Delta_X)\,.
\end{gather}
\end{definition}
For functions that vanish outside of $\SS_n$, the notions of Möbius transform
and graded Möbius transform coincide, while this is not the case in general.

The generalization of the summation formula~\eqref{eq:10} stands in the
following result.

\begin{theorem}
  \label{thr:3}
Let $f,h:\bbr n\to\bbR$ be two functions. Then $h$ is the graded Möbius
transform of $f$ if and only if
\begin{gather}
\label{eq:11}
\forall x\in\bbr n\quad f(x)=\sum_{y\in\bbr n[x]}h(x \cdot y)\,.
\end{gather}
\end{theorem}

Note that, in formula~(\ref{eq:11}), the braids $y\in\bbr n[x]$ may have
normal forms that differ completely from that of~$x$. This relates with
Remark~\ref{rem:1}.

\begin{proof}
For a generic braid $x\in\bbr n$ of height $k=\height(x)$, we denote by
$(x_1,\ldots,x_k)$ the Garside decomposition of $x$. Observe that, for all
$x, y, z\in \bbr n$, we have:
\begin{gather}
 \label{eq:equivalence}
y \in \bbr n[x] \wedge z \in \bbr n[x \cdot y] \iff
y \cdot z \in \bbr n[x].
 \end{gather}
 Indeed, $y \in \bbr n[x] \wedge z \in \bbr n[x \cdot y] \iff \tau(x) = \tau(x \cdot y) = \tau(x \cdot y \cdot z) \iff
\tau(x) = \tau(x \cdot y \cdot z) \iff y \cdot z \in \bbr n[x]$.

Hence, if $h$ is the graded Möbius transform of $f$, then:
\begin{align*}
\sum_{y\in\bbr n[x]} h(x \cdot y) &=\sum_{y \in \bbr n} \sum_{X \subseteq \Sigma} (-1)^{|X|} \un{y \in \bbr n[x]} \un{\Delta_X \in \bbr n[x \cdot y]} f(x \cdot y \cdot \Delta_X)
&&\text{by~\eqref{eq:4**}}\\
&= \sum_{v \in \bbr n} \sum_{y \in \bbr n} \sum_{X \subseteq \Sigma} (-1)^{|X|} \un{v \in \bbr n[x]} \un{v = y \Delta_X} f(x \cdot v)
&&\hspace{-0.9cm}\text{by~\eqref{eq:equivalence} with $z=\Delta_X$}\\
&= \sum_{v \in \bbr n} \sum_{X \subseteq \Sigma} (-1)^{|X|} \un{v \in \bbr n[x]} \un{\Delta_X \leqr v} f(x \cdot v) \\
&= \sum_{v \in \bbr n} \Bigl(\sum_{X \subseteq \Dl(v^*)} (-1)^{|X|}\Bigr) \un{v \in \bbr n[x]} f(x \cdot v) \\
&= f(x).
\end{align*}

Conversely, if~\eqref{eq:11} holds, then:
\begin{align*}
\sum_{X \subseteq \Sigma} (-1)^{|X|} &\un{\Delta_X \in \bbr n[x]} f(x\cdot \Delta_X) \\
&= \sum_{X \subseteq \Sigma} \sum_{z \in \bbr n}(-1)^{|X|} \un{\Delta_X \in \bbr n[x]} \un{z \in \bbr n[x \cdot \Delta_X]} h(x \cdot \Delta_X \cdot z) \\
&= \sum_{u \in \bbr n} \sum_{X \subseteq \Sigma} \sum_{z \in \bbr n}(-1)^{|X|} \un{u \in \bbr n[x]} \un{u = \Delta_X \cdot z} h(x \cdot u) \\
&= \sum_{u \in \bbr n[x]} \sum_{X \subseteq \Sigma} (-1)^{|X|} \un{\Delta_X \leql u} h(x \cdot u) \\
&= \sum_{u \in \bbr n[x]} \Bigl(\sum_{X \subseteq \Dl(u)} (-1)^{|X|} \Bigr) h(x \cdot u) \\
&= h(x).
\end{align*}
This completes the proof.
\end{proof}

% Consider the graph $(\SS_n,\to)$. The above discussion shows that
% $\bbr n$ is in bijection with the set of infinite paths in
% $(\SS_n,\to)$ that eventually hit $\unit$, and then necessarily
% stay in $\unit$.

\subsubsection{Additional properties of Möbius transforms}
\label{sec:additional-results}

Finally, we state in this subsection a couple of lemmas which we will use in
next section for the probabilistic study.

\begin{lemma}
  \label{lem:2}
  For $p$ a real number, let $f:\SS_n\to\bbR$ be defined by
  $f(x)=p^{\abs{x}}$, and let $h:\SS_n\to\bbR$ be the Möbius transform
  of~$f$. Let also $g:\SS_n\to\bbR$ be defined by:
\begin{gather}
\label{eq:17}
g(x)=\sum_{y\in\SS_n} \un{x \to y}h(y)\,.
\end{gather}
Then the identity $h(x)=f(x)g(x)$ holds for all $x\in\SS_n$\,.
\end{lemma}

\begin{proof}
  Let $P=\mathcal{P}(\Sigma)$, and consider the two functions
  $F,G:P\to\bbR$ defined, for $A\in P$, by:
\begin{align*}
F(A) &= \sum_{I \in P} (-1)^{\abs{I}} \un{I \subseteq \Dl(\Delta_{\Sigma \setminus A})} f(\Delta_I)\,,&
G(A) &= \sum_{y \in \SS_n} \un{\Dl(y) \cap \Dl(\Delta_{\Sigma\setminus A}) = \emptyset} h(y)\,.
\end{align*}
Then we claim that $F=G$.

Let us prove the claim. For every $I\in P$ and for every $y\in\SS_n$, we
have:
\[I\subseteq\Dl(y)\iff\Delta_I\leql y \iff \Dl(\Delta_I) \subseteq \Dl(y).\]
Therefore, according to the Möbius summation formula~\eqref{eq:10}, we have:
\[f(\Delta_I)=\sum_{y\in\SS_n}\un{I\subseteq\Dl(y)}h(y).\]
Reporting the right hand member above in the sum defining $F(A)$, and
inverting the order of summations, yields:
\[F(A) =\sum_{y\in\SS_n}\Bigl(\sum_{I\in
    P}(-1)^{\abs{I}}\un{I\subseteq \Dl(\Delta_{\Sigma\setminus A})
    \cap \Dl(y)}\Bigr)h(y)
  =\sum_{y\in\SS_n}\un{\Dl(\Delta_{\Sigma\setminus A}) \cap \Dl(y) =
    \emptyset}h(y) = G(A),\] which proves the claim.

Now observe that, for every $x\in\SS_n$, we have
\[x \cdot \Delta_{\Sigma \setminus \Dr(x)} = \bigveel \{x \cdot \sigma \tq \sigma \in \Sigma \setminus \Dr(x)\} \leq \bigveel \SS_n = \Delta_\Sigma,\]
and therefore $\Dl(\Delta_{\Sigma \setminus \Dr(x)}) = \Sigma \setminus
\Dr(x)$. This proves that
\begin{align}
  \label{eq:24}
  x \cdot \Delta_I \in \SS_n\  &
                                 \iff\ \makebox[6em]{$I \cap \Dr(x) = \emptyset$}
                                 \iff\   I \subseteq \Dl(\Delta_{\Sigma    \setminus \Dr(x)})\\
  \label{eq:26}
  x \to y\  & \iff\  \makebox[6em]{$\Dl(y) \subseteq \Dr(x)$}
              \iff\   \Dl(y) \cap \Dl(\Delta_{\Sigma \setminus \Dr(x)}) = \emptyset.
\end{align}
% \[\arraycolsep=5pt\begin{array}{lllll}x \cdot \Delta_I \in \SS_n & \iff & I \cap \Dr(x) = \emptyset & \iff & I \subseteq \Dl(\Delta_{\Sigma \setminus \Dr(x)}) \\
% x \to y & \iff & \Dl(y) \subseteq \Dr(x) & \iff & \Dl(y) \cap \Dl(\Delta_{\Sigma \setminus \Dr(x)}) = \emptyset.
% \end{array}\]
Hence, using the multiplicativity of $f$, we have simultaneously
\begin{align*}
h(x)&=\sum_{I\in P}(-1)^{\abs{I}} \un{I \subseteq \Dl(\Delta_{\Sigma \setminus \Dr(x)})} f(x\cdot\Delta_I) =  f(x)F\left(\Dr(x)\right) \\
\text{and }g(x)&=\sum_{y\in\SS_n}\un{\Dl(y) \cap \Dl(\Delta_{\Sigma\setminus\Dr(x)} = \emptyset}h(y)=G(\Dr(x))
\end{align*}
% \[\arraycolsep=0.5pt\def\arraystretch{1.2}\begin{array}{l}
% h(x)=\sum_{I\in P}(-1)^{\abs{I}} \un{I \subseteq \Dl(\Delta_{\Sigma \setminus \Dr(x)})} f(x\cdot\Delta_I) =  f(x)F\left(\Dr(x)\right) \\
% g(x)=\sum_{y\in\SS_n}\un{\Dl(y) \cap \Dl(\Sigma\setminus\Dr(x) = \emptyset}h(y)=G(\Dr(x)).\end{array}\]
Since $F=G$, it implies $h(x)=f(x)g(x)$, which completes the proof of the
lemma.
\end{proof}

\begin{lemma}
\label{lem:3b} Let $(x_1,\ldots,x_k)$ be the Garside decomposition of a braid
$x \in \bbr n$ and let $X$ be a subset of\/~$\Sigma$. We have:
\begin{gather}
\label{eq:3b}
\Delta_X \in \bbr n[x] \iff \Delta_X \in \bbr n[x_k].
\end{gather}
\end{lemma}

\begin{proof}
  The result is immediate if $x = \unit$.  Moreover, if $x \neq \unit$
  and if $\Delta_X \in \bbr n[x_k]$, we observe that $x_k \Delta_X$ is
  a simple braid, and therefore that
  $x_1 \cdot \ldots \cdot x_{k-1} \cdot (x_k \Delta_X)$ is a
  factorization of $x$ into $k$ simple braids, whence
  $\height(x \Delta_X) \leq \height(x)$.  Since the Garside length is
  non-decreasing for $\leql$, it follows that
  $\Delta_X \in \bbr n[x]$.

  Conversely, if $x \neq \unit$ and if $\Delta_X \notin \bbr n[x_k]$,
  since the set $\SS_n$ is closed under~$\veel$, there must exist
  some generator $\sigma \in X \setminus \bbr n[x_k]$.  Hence, we have
  $x_1 \to \ldots x_k \to \sigma$, and therefore
  $\height(x \cdot \Delta_X) \geq \height(x \cdot \sigma) = k+1$,
  i.e.\ $\Delta_X \notin \bbr n[x]$.
\end{proof}

\begin{corollary}
 \label{cor:cor}
 Let $f:\bbr n\to\bbR$ be the function defined by $f(x)=p^{|x|}$. Then
 the graded Möbius transform $h:\bbr n\to\bbR$ of $f$ satisfies the
 following property:
 \begin{gather}
  \label{eq:mobtransf}
  h(x)=p^{|x_1|+\ldots+|x_{k-1}|} h(x_k)\,,
 \end{gather}
where $(x_1,\ldots,x_k)$ is the Garside decomposition of~$x$.
\end{corollary}

\begin{proof}
  It follows directly from the definition~\eqref{eq:4} of the graded
  Möbius transform, together with Lemma~\ref{lem:3b}.
\end{proof}

\section{Uniform measures on braid monoids}
\label{sec:unif-meas-braid}

We are now equipped with adequate tools to study uniform measures on braids.
Consider the following (vague) questions: how can we pick a braid uniformly
at random? How can we pick a large braid uniformly at random? What are the
characteristics of such random braids?

Since there are countably many braids, these questions cannot be given
immediately a consistent meaning. However, for each fixed integer
$k\geq0$, there are finitely many braids of size~$k$, and it is thus
meaningful to pick uniformly a braid of size $k$ at random. Please
notice the difference between picking a braid of size $k$ uniformly at
random, and picking a word uniformly in~$\Sigma^k$ and then
considering the braid it induces. The later corresponds to the uniform
random walk on~$\bbr n$\,, but not the former.

A possible way of picking a braid at random is thus the following:
first pick the size $k$ at random, and then pick a braid uniformly
among those of size~$k$. The problem remains of how to draw $k$ in a
``natural'' way. It is the topic of this section to demonstrate that
there is indeed a natural family, indexed by a real parameter~$p$, of
conducting this random procedure.  Furthermore, the parameter $p$ is
bound to vary in the interval $(0,q_n)$, where $q_n$ is the root of
$H_n(t)$ introduced earlier; and letting $p$ tend to~$q_n$ by inferior
values, the distributions induced on braids weight  large braids more
and more, such that at the limit we obtain a natural uniform measure
on ``infinite braids''. In turn, we shall derive in the next section
information on large random braids, that is to say, on random braids
of size $k$ when $k$ is large enough, based on the notion of uniform
measure on infinite braids.

\subsection{Generalized braids}
\label{subsubsec:gen_braids}

Considering the extended Garside decomposition of braids, one sees that
elements of $\bbr n$ are in bijection with \emph{infinite paths
  in $(\SS_n, \to)$ that eventually hit~$\unit$}. Therefore, it is
natural to define a compactification $\bbrbar n$ as the set of all infinite
paths in this graph. As a subset of $\SS_n^{\bbN^*}$, it is endowed with a
canonical topology, for which it is compact. Moreover, the restriction of
this topology to $\bbr n$ is the discrete topology, and $\bbrbar n$ is the
closure of~$\bbr n$. This is the set of \emph{generalized braids}. We endow
the set $\bbrbar n$ with its Borel \slgb. All measures we shall consider on
$\bbrbar n$ will be finite and Borelian.

We may refer to elements of $\bbr n$ as \emph{finite braids}, to emphasize
their status as elements of~$\bbrbar n$\,.  We define the \emph{boundary}
$\bdbbr n$ by:
\[
  \bdbbr n=\bbrbar n\setminus\bbr n\,.
\]
Elements in $\bdbr n$ correspond to infinite paths in $(\SS_n,\to)$ that
never hit~$\unit$, we may thus think of them as \emph{infinite braids}.

If $(x_1,\ldots,x_p)$ is a finite path in the graph $(\SS_n,\to)$, the
corresponding cylinder set $\DD_{(x_1,\ldots,x_p)}$ is defined as the set of
paths starting with vertices $(x_1,\ldots,x_p)$. Cylinder sets are both open
and closed, and they generate the topology on~$\bbrbar n$\,.

\begin{definition}
  \label{def:2}
For $x\in\bbr n$ of Garside decomposition $(x_1,\ldots,x_p)$, we define the
\emph{Garside cylinder}, and we denote by~$\CC_x$\,, the cylinder subset of
$\bbrbar n$ given by $\CC_x=\DD_{(x_1,\ldots,x_p)}$\,.
\end{definition}

Garside cylinders only reach those cylinders sets of the form
$D=\DD_{(x_1,\ldots,x_p)}$ with $x_p\neq\unit$\,. But if $x_p=e$\,, then $D$
reduces to the singleton~$\{x\}$, with $x=x_1\cdot\ldots\cdot x_p$\,. And
then, denoting by $q$ the greatest integer with $x_q\neq\unit$ and writing
$y=x_1\cdot\ldots\cdot x_q$\,, one has:
\[
\{x\}=\CC_{y}\setminus\bigcup_{z\in\SS_n\setminus\{\unit\}\tq
  x_q\to z}\CC_{y\cdot z}\,.
\]
It follows that \emph{Garside cylinders generate the topology
  on~$\bbrbar n$}\,, which implies the following result.

\begin{proposition}
  \label{prop:5}
Any finite measure on the space $\bbrbar n$ of generalized braids is entirely
determined by its values on Garside cylinders. In other words, if $\nu$ and
$\nu'$ are two finite measures on $\bbrbar n$ such that
$\nu(\CC_x)=\nu'(\CC_x)$ for all $x\in\bbr n$\,, then $\nu=\nu'$.
\end{proposition}

\begin{proof}
  We have already seen that Garside cylinders generate the topology,
  and thus the Borel \slgb\ of~$\bbrbar n$\,. The collection of
  Garside cylinders, augmented with the empty set, is obviously stable by intersection:
  \[
    \forall x,y\in\bbr n\quad \text{ either }\CC_x \subseteq \CC_y
    \text{ or } \CC_x \supseteq \CC_y \text{ or } \CC_x \cap \CC_y = \emptyset\,,
  \]
and forms thus a so-called $\pi$-system. The result follows from classical
measure theory~\cite[Th.~3.3]{billingsley95}.
\end{proof}

Garside cylinders are very natural from the point of view of the normal
forms, however they are somewhat unnatural from the algebraic point of view
as they discard most of the divisibility information
(\textit{cf.}~Remark~\ref{rem:1}). A more natural notion is that of
\emph{visual cylinder}, which corresponds, for a given finite braid $x\in\bbr
n$, to the subset of those generalized braids which are ``left divisible''
by~$x$. It will be useful to differentiate between generalized braids and
infinite braids, therefore we introduce both the \emph{full visual cylinder
$\Up x$} and the \emph{visual
  cylinder~$\upp x$}, as follows:
\begin{align*}
  \Up x&=\text{Closure}\bigl(\{x\cdot z\ :\ z\in\bbr n\}\bigr)\,,&
\upp x&=\Up x\cap\bdbbr n\,,
\end{align*}
where $\text{Closure}(A)$ denotes the topological closure of the set~$A$.

The relationship between Garside cylinders and visual cylinders is given by
the following result.

\begin{lemma}
\label{lem:full-visual-is-union-of-garside} For each finite braid $x \in \bbr
n$, the full visual cylinder $\Up x$ is the following disjoint union of
Garside cylinders:
\begin{equation}
\label{eq:defUp}
  \Up x = \bigcup_{y \in \bbr n[x]} \CC_{x \cdot y}.
\end{equation}
\end{lemma}

\begin{proof}
  We first observe that
  $\Up x \cap \bbr n = \bigcup_{y \in \bbr n[x]} (\bbr n \cap \CC_{x
    \cdot y})$.
  Indeed, the $\supseteq$ inclusion is obvious, while the converse one
  is a consequence of points~\ref{pro:garside-2}
  and~\ref{pro:garside-3} of
  Proposition~\ref{proposition:dehornoy2013foundations}.  Since
  $\Up x$ and $\bigcup_{y \in \bbr n[x]} \CC_{x \cdot y}$ are the
  respective topological closures of $\Up x \cap \bbr n$ and of
  $\bigcup_{y \in \bbr n[x]} (\bbr n \cap \CC_{x \cdot y})$
  in~$\bbrbar n$, the result follows.
\end{proof}

Hence, $\Up x$, as a finite union of Garside cylinders, is also open and
closed in~$\bbrbar n$. In the same way, $\upp x$ is open and closed
in~$\bdbbr n$.

%
% This pedestrian approach to $\bbrbar n$ is very convenient for explicit
% descriptions and computations, but it is lacking from the point of view of
% canonicity and generality, as we have relied on a specific normal form. A
% more conceptual approach to this very same object, in which divisibility
% relations are more central and which avoids the use of normal forms, is
% described in Appendix~\ref{subsec:canonical}.

\subsection{Studying finite measures on generalized braids via
  the graded Möbius transform}
\label{subsec:measures}

In this subsection, we study finite measures on the set $\bbrbar n$ of
generalized braids.

Assume that $\nu$ is some finite measure on~$\bbrbar n$\,. Then for practical
purposes, we are mostly interested in the values of $\nu$ on Garside
cylinders~$\nu(\CC_x)$. However, most natural measures will enjoy good
properties with respect to the full visual cylinders~$\Up x$, which is not
surprising as these sets are most natural from the point of view of
divisibility properties. For instance, the limits $\nu$ of uniform measures
on the set of braids of length $k$ will satisfy $\nu(\Up x) = p^{\abs{x}}$
for some~$p$, see Definition~\ref{def:1} and Theorem~\ref{thr:9} below.

Henceforth, to understand these measures, we need to relate $\nu(\CC_x)$ and
$\nu(\Up x)$ in an explicit way, and this is where the graded Möbius
transform of Subsection~\ref{sec:grade-mobi-transf} plays a key role, as
shown by Proposition~\ref{cor:1} below. In turn, Proposition~\ref{cor:1}
provides a nice probabilistic interpretation of the graded Möbius transform.

\begin{proposition}
  \label{cor:1}
  Let $\nu$ be a finite measure on~$\bbrbar n$\,. Let
  $f:\bbr n\mapsto\bbR$ be defined by $f(x)=\nu(\Up x)$\,, and let
  $h:\bbr n\mapsto\bbR$ be the graded Möbius transform of~$f$. Then,
  for every integer $k\geq1$ and for every finite braid $y$ of
  height~$k$, holds:
\begin{gather}
  \label{eq:16}
  \nu(\CC_y)=h(y).
\end{gather}
\end{proposition}
\begin{proof}
The decomposition~\eqref{eq:defUp} of a full visual cylinder as a disjoint
union of Garside cylinders shows that
\begin{equation*}
  \nu(\Up x) = \sum_{y\in \bbr n[x]} \nu(\CC_{x\cdot y}).
\end{equation*}
Thus, the characterization~\eqref{eq:11} of the graded Möbius transforms
shows that $y\mapsto \nu(\CC_y)$ is the graded Möbius transform of $x \mapsto
\nu(\Up x)$, as claimed.
\end{proof}

\begin{corollary}
  \label{cor:3}
  A finite measure $\nu$ on $\bbrbar n$ is entirely determined by its
  values $\nu(\Up x)$ on the countable collection of full visual
  cylinders.
\end{corollary}
\begin{proof}
  According to Proposition~\ref{prop:5}, a finite measure $\nu$ is
  entirely determined by its values on Garside cylinders. Hence the
  result follows from Proposition~\ref{cor:1}.
\end{proof}

\subsection{Uniform measures}
\label{sec:uniform-measures}

Our ultimate goal is to understand the uniform measure $\mu_{n,k}$ on braids
in $\bbr n$ of a given length~$k$, when $k$ tends to infinity. We will see
below in Theorem~\ref{thr:9} that this sequence of measures converges to a
measure on $\bdbbr n$ which behaves nicely on the visual cylinders~$\upp x$
(this is not surprising as these are the natural objects from the point of
view of the monoid structure on~$\bbr n$).

Therefore, it is good methodology to study the general class of measures
which do behave nicely on \emph{full} visual cylinders.  Our usual
conventions and notations are in force throughout this subsection, in
particular concerning $\bbr n$ which may be either $\br n$ or~$\dbr n$\,.

\begin{definition}
  \label{def:1}
A \emph{uniform measure for braids} of parameter $p\geq0$ is a probability
measure $\nu_p$ on $\bbrbar n$ satisfying:
\begin{gather*}
  \forall x\in\bbr n\quad \nu_p(\Up x)=p^{\abs{x}}\,.
\end{gather*}
\end{definition}

Although not apparent from this definition, we will see in Theorem~\ref{thr:1}
below that such a measure either weights $\bbr n$ or~$\bdbbr n$, but not
both. Theorem~\ref{thr:1} will describe quite precisely all uniform measures.
It will allow us to define the \emph{uniform measure at infinity} as the
unique non trivial uniform measure supported by the boundary~$\bdbr n$\,, see
Definition~\ref{def:7}. A realization result for uniform measures will be the
topic of Subsection~\ref{sec:mark-chain-real}.

Before coming to the theorem, we state a key lemma.

\begin{lemma}
\label{lem:3} Let $\nu$ be a uniform measure of parameter $p<1$. Assume that
$\nu$ is concentrated at infinity, \emph{i.e.}, $\nu(\bdbbr n)=1$. Then
$H_n(p)=0$.

Furthermore, let $B=(B_{x,x'})$ be the non-negative matrix indexed by pairs
of simple braids $(x,x')$ such that $x,x'\in\SS_n\setminus\{\unit,\Delta\}$,
and defined by:
  \[
    B_{x,x'}=\un{x\to x'}p^{\abs{x'}}\,.
  \]
  Then $B$ is a primitive matrix of spectral radius~$1$. The Perron
  right eigenvector of $B$ is the restriction to
  $\SS_n\setminus\{\unit,\Delta\}$ of the vector $g$ defined
  in\/~\eqref{eq:17}.
\end{lemma}

\begin{proof}
Let $f(x)=p^{\abs{x}}$\,, and let $h:\SS_n\to\bbR$ be the graded Möbius
transform of~$f$.

According to Proposition~\ref{cor:1}, we have
$h(\unit)=\nu(\CC_\unit)=\nu(\{\unit\})$. Since it is assumed that $\nu(\bbr
n)=0$, it follows that $h(\unit)=0$. But $H_n(p)=h(\unit)$, as previously
stated in~\eqref{eq:19}, hence $H_n(p)=0$, proving the first claim of the
lemma.

Let $g$ be defined on $\SS_n$ as in~\eqref{eq:17}, and let $\widetilde g$ be
the restriction of $g$ to $\SS_n\setminus\{\unit,\Delta\}$.  It follows from
Lemma~\ref{lem:2} that $h(x)=p^{\abs{x}}g(x)$ holds on $\SS_n$. Therefore the
computation of $B\widetilde g$ goes as follows, for
$x\in\SS_n\setminus\{\unit,\Delta\}$:
\[
(B\widetilde g)_{x} =\sum_{y\in\SS_n\setminus\{\unit,\Delta\}}\un{x \to y}p^{\abs{y}}g(y)=\sum_{y \in\SS_n\setminus\{\unit,\Delta\}}\un{x \to y}h(y).\]

But $h(\unit)=0$ on the one hand; and on the other hand, $x\to\Delta$ does
not hold since $x\neq\Delta$. Hence the above equality rewrites as:
\[(B\widetilde g)_x =\sum_{y\in\SS_n} \un{x \to y} h(y)=\widetilde g(x).\]

We have proved that $\widetilde g$ is right invariant for~$B$. Let us
prove that $\widetilde g$ is non identically zero. Observe that $h$ is
non-negative, as a consequence of Proposition~\ref{cor:1}. Therefore
$g$ is non-negative as well. If $\widetilde g$ were identically zero
on $\SS_n\setminus\{\unit,\Delta\}$, so would be $h$ on
$\SS_n\setminus\{\Delta\}$. The Möbius summation formula~\eqref{eq:1}
would imply that $f$ is constant, equal to $f(\Delta)$ on $\SS_n$\,,
which is not the case since we assumed $p\neq1$. Hence $\widetilde g$
is not identically zero.

But $B$ is also aperiodic and irreducible, hence primitive. Therefore
Perron-Frobenius theory~\cite[Chapter~1]{seneta81} implies that $\widetilde
g$ is actually \emph{the} right Perron eigenvector of~$B$, and $B$ is thus of
spectral radius~$1$.
\end{proof}

\begin{theorem}
  \label{thr:1}
  For each braid monoid $\bbr n$, uniform measures $\nu_p$ on $\bbrbar
  n$ are parametrized by the parameter $p$ ranging exactly over the
  closed set of reals $[0,q_n]\cup\{1\}$.
\begin{enumerate}
\item\label{item:4} For $p=0$, $\nu_0$ is the Dirac measure at\/~$\unit$.
\item\label{item:1} For $p=1$, $\nu_1$ is the Dirac measure on the element
    $\Delta^\infty$ defined by its infinite Garside decomposition:
    $(\Delta\cdot\Delta\cdot\ldots)$\,.
\item\label{item:6} For $p\in(0,q_n)$, the support of $\nu_p$ is~$\bbr n$,
    and it is equivalently characterized by:
  \begin{align}
    \label{eq:13}
\nu_p\bigl(\{x\}\bigr)&=H_n(p)\cdot p^{\abs{x}}&&\text{or}&
\nu_p(\Up x)&=p^{\abs{x}}
  \end{align}
for $x$ ranging over~$\bbr n$\,.
\item\label{item:7} For $p=q_n$, the support of $\nu_{q_n}$ is~$\bdbbr n$,
    and it is characterized by:
  \begin{gather}
    \label{eq:14}
    \forall x\in\bbr n\quad\nu_{q_n}(\upp x)=q_n^{\abs{x}}\,.
  \end{gather}
%\]
\end{enumerate}
\end{theorem}

It follows from this statement that, except for the degenerated
measure~$\nu_1$\,, there exists a unique uniform measure on the
boundary~$\bdbbr n$\,. It is thus natural to introduce the following
definition.

\begin{definition}
  \label{def:7}
  The uniform measure on $\bdbbr n$ which is characterized by
  $\nu_{q_n}(\upp x)=q_n^{\abs{x}}$ for $x\in\bbr n$, is called the
  \emph{uniform measure at infinity}.
\end{definition}

\begin{proof}[Proof of Theorem~\ref{thr:1}.]
  The statements~\ref{item:4}--\ref{item:7} contains actually three
  parts: the existence of~$\nu_p$ for each $p\in[0,q_n]\cup\{1\}$, the
  uniqueness of the measures satisfying the stated characterizations,
  and that $[0,q_n]\cup\{1\}$ is the only possible range for~$p$.

  \emph{Existence and uniqueness of $\nu_p$ for
    $p\in[0,q_n]\cup\{1\}$.}\quad The cases $p=0$ and $p=1$
  (points~\ref{item:4} and~\ref{item:1}) are trivial. For
  $p\in(0,q_n)$ (point~\ref{item:6}), let $\nu_p$ be the discrete
  distribution on $\bbr n$ defined by the left hand side
  of~\eqref{eq:13}. Since $p<q_n$, the series $G_n(p)$ is convergent,
  and it implies that the following formula is valid in the field of
  real numbers:
  \[
    G_n(p)\cdot H_n(p)=1\,.
  \]
It implies in particular that $H_n(p)>0$, and thus:
\[
  \sum_{x\in\bbr n}\nu_p\bigl(\{x\}\bigr)=1\,,
\]
and therefore $\nu_p$ is a probability distribution on~$\bbr n$\,.

It remains to prove that $\nu_p$ is indeed uniform with parameter~$p$. Since
$\bbr n$ is left cancellative, we notice that, for each $x\in\bbr n$, the
mapping $y\in\bbr n\mapsto x\cdot y$ is a bijection of $\bbr n$ onto $\Up
x\,\cap\bbr n$. Whence:
\begin{align*}
  \nu_p(\Up x)&=H_n(p)\cdot p^{\abs{x}}\cdot\Bigl(\sum_{y\in\bbr n}p^{\abs{y}}\Bigr)=p^{\abs{x}}\,.
\end{align*}

Conversely, if $\nu$ is a probability measure on~$\bbrbar n$ such that
$\nu(\Up x)=p^{\abs{x}}$\,, then $\nu$ and $\nu_p$ agree on full visual
cylinders, hence $\nu=\nu_p$ according to Corollary~\ref{cor:3}.

We now treat the case of point~\ref{item:7}, corresponding to
$p=q_n$\,. For this, let $(p_j)_{j\geq1}$ be any sequence of reals
$p_j<q_n$ such that $\lim_{j\to\infty}p_j=q_n$\,, and such that
$(\nu_{p_j})_{j\geq1}$ is a weakly convergent sequence of probability
measures. Such a sequence exists since $\bbrbar n$ is a compact metric
space. Let $\nu$ be the weak limit of~$(\nu_{p_j})_{j\geq1}$\,.

Obviously, for each braid $x$ fixed:
\begin{align*}
  \lim_{j\to\infty} \nu_{p_j}(\Up x)=q_n^{\abs{x}}\,.
\end{align*}
But $\Up x$ is both open and closed in~$\bbrbar n$, it has thus an empty
topological boundary. The Portemanteau theorem~\cite{billingsley95} implies
that the above limit coincides with~$\nu(\Up x)$\,, hence $\nu(\Up
x)=q_n^{\abs{x}}$ for all $x\in\bbr n$\,. The same reasoning applied to every
singleton~$\{x\}$, for $x\in\bbr n$\,, yields:
\begin{align*}
  \nu(\{x\})=\lim_{j\to\infty}\nu_{p_j}(\{x\})=\lim_{j\to\infty}\frac{p_j^{\abs{x}}}{G_n(p_j)}=0\,,
\end{align*}
the later equality since $\lim_{t\to q_n^-}G_n(t)=+\infty$\,. Since $\bbr n$
is countable, it follows that $\nu$ puts weight on $\bdbbr n$ only, and thus:
\[
  \forall x\in\bbr n\quad\nu(\upp x)=\nu(\Up x)=q_n^{\abs{x}}\,.
\]

If $\nu'$ is a probability measure on $\bbrbar n$ satisfying $\nu'(\upp
x)=q_n^{\abs{x}}$ for every $x\in\bbr n$\,, then we observe first that $\nu'$
is concentrated on the boundary, since $\nu'(\bdbbr n)=\nu'(\upp\unit)=1$.
And since $\nu$ and $\nu'$ coincide on all visual cylinders~$\upp x$, for $x$
ranging over~$\bbr n$, it follows from Corollary~\ref{cor:3} that $\nu=\nu'$.

\emph{Range of~$p$.}\quad It remains only to prove that, if $\nu$ is a uniform
probability measure on $\bbrbar n$ of parameter~$p$, then $p=1$ or $p\leq
q_n$\,. Seeking a contradiction, assume on the contrary that $p>q_n$ and
$p\neq1$ holds.

We first show the following claim:
\begin{gather}
  \label{eq:23}
\nu(\bdbbr n)=1\,.
\end{gather}

Assume on the contrary $\nu(\bdbbr n)<1$, hence $\nu(\bbr n)>0$.  Then we
claim that the inclusion-exclusion principle yields:
\begin{gather}
\label{eq:20}
\forall x\in\bbr n\quad\nu\bigl(\{x\}\bigr)=H_n(p)\cdot p^{\abs{x}}\,.
\end{gather}

Indeed, for any braid $x\in\bbr n$\,, the singleton $\{x\}$ decomposes as:
\[\{x\}=\Up x\setminus\bigcup_{\sigma \in \Sigma}\Up(x\cdot\sigma)\]
and therefore:
\[\nu\left(\{x\}\right) = \sum_{I \subseteq \Sigma} (-1)^{|I|} \nu \Bigl(\bigcap_{\sigma \in I} \Up(x \cdot \sigma)\Bigr)
= \sum_{I \subseteq \Sigma} (-1)^{|I|} \nu \bigl(\Up(x \cdot \Delta_I)\bigr)\,.
\]

Note that the above equality is valid for any finite measure on~$\bbrbar
n$\,. Since $\nu$ is assumed to be uniform of parameter~$p$, it specializes
to the following:
\[
\nu\left(\{x\}\right)=\sum_{I\subseteq\Sigma}(-1)^{|I|}p^{\abs{x}+\abs{\Delta_I}}=p^{\abs{x}}\cdot H_n(p)\,,
\]
given the form~(\ref{eq:2}) for~$H_n(p)$. This proves our
claim~(\ref{eq:20}).

Together with $\nu(\bbr n)>0$, it implies $H_n(p)>0$. Consequently, summing
up $\nu(\{x\})$ for $x$ ranging over~$\bbr n$ yields $G_n(p)<\infty$. Hence
$p<q_n$\,, which is a contradiction since we assumed $p>q_n$\,. This proves
the claim~\eqref{eq:23}.

Next, consider the two matrices $B$ and $B'$ indexed by all braids $x \in
\SS_n \setminus \{\unit,\Delta\}$ and defined by:
\[B_{x,x'}=\un{x\to x'}p^{\abs{x'}} \quad\text{and}\quad B'_{x,x'}=\un{x\to x'}q_n^{\abs{x'}}.\]
They are both non-negative and primitive, and of spectral radius $1$
according to Lemma~\ref{lem:3} (which applies since $p\neq1$ by assumption).
According to Perron-Frobenius theory~\cite[Chapter~1]{seneta81}, there cannot
exist a strict ordering relation between them. Yet, this is implied by
$p>q_n$\,, hence the latter is impossible. The proof is complete.
\end{proof}

\begin{remark}[Multiplicative measures]
\label{rem:7}

Since the length of braids is additive, any uniform measure is
multiplicative, \emph{i.e.}, it satisfies: $\nu_p(\Up(x\cdot y))=\nu_p(\Up
x)\cdot\nu_p(\Up y)$.

Conversely, assume that $\nu$ is a multiplicative probability measure
on~$\bbrbar n$. Then $\nu$ is entirely determined by the values
$p_\sigma=\nu(\Up \sigma)$ for $\sigma \in \Sigma$.

If $\bbr n = \br n$, let us write $p_i$ instead of $p_{\sigma_i}$. The braid
relations $\sigma_i\cdot\sigma_{i+1}\cdot\sigma_i =
\sigma_{i+1}\cdot\sigma_i\cdot\sigma_{i+1}$ entail: $p_i
p_{i+1}(p_i-p_{i+1})=0$. Hence, if any two consecutive $p_i,p_{i+1}$ are non
zero, they must be equal. Removing the generators $\sigma_i$ for which
$p_i=0$, the braid monoid splits into a direct product of sub-braid monoids,
each one equipped with a uniform measure.

Similarly, if $\bbr n = \dbr n$, let us write $p_{i,j}$ instead
of~$p_{\sigma_i,\sigma_j}$.  Then the dual braid relations
$\sigma_{i,j} \cdot \sigma_{j,k} = \sigma_{j,k} \cdot
\sigma_{k,i}=\sigma_{k,i}\cdot\sigma_{i,j}$
(if $i < j < k$) yield the following three relations:
$p_{j,k}(p_{i,j}-p_{k,i})= 0$, $p_{i,k}(p_{i,j}-p_{j,k})=0$ and
$p_{i,j}(p_{j,k}-p_{i,k})=0$. Therefore the following implication
holds for all $i<j<k$:
$(p_{i,j}>0\wedge p_{j,k}>0)\implies p_{i,j}=p_{j,k}=p_{i,k}$.
Removing the generators $\sigma_{i,j}$ for which $p_{i,j} = 0$, the
dual braid monoid splits thus into a direct product of sub-dual braid
monoids, each one equipped with a uniform measure.

Therefore, without loss of generality, the study of multiplicative measures
for braid monoids reduces to the study of uniform measures. This contrasts with
other kinds of monoids, such as heap monoids, see~\cite{abbes15b} and the
discussion in Section~\ref{se-ext}.
\end{remark}

\subsection{Markov chain realization of
  uniform measures}
\label{sec:mark-chain-real}

Recall that generalized braids $\xi\in\bbrbar n$ are given by infinite
sequences of linked vertices in the graph $(\SS_n,\to)$. For each integer
$k\geq1$, let $X_k(\xi)$ denote the $k^{\text{th}}$ vertex appearing in an
infinite path $\xi\in\bbrbar n$\,. This defines a sequence of measurable
mappings
\[
X_k:\bbrbar n \to\SS_n\,,
\]
which we may interpret as random variables when equipping $\bbrbar n$ with a
probability measure, say for instance a uniform measure~$\nu_p$\,.

It turns out that, under any uniform measure~$\nu_p$\,, the process
$(X_k)_{k\geq1}$ has a quite simple form, namely that of a Markov chain. This
\emph{realization result} is the topic of the following theorem (the trivial
cases $p=0$ and $p=1$ are excluded from the discussion).

\begin{theorem}
  \label{thr:2}
  Let $p\in(0,q_n]$, and let $\nu_p$ be the uniform measure of
  parameter $p$ on~$\bbrbar n$\,. Let $h:\SS_n\to\bbR$ be the Möbius
  transform of $x\in\SS_n\mapsto p^{\abs{x}}$\,.
\begin{enumerate}
\item\label{item:8} Under~$\nu_p$\,, the process $(X_k)_{k\geq1}$ of simple
    braids is a Markov chain, taking its values in $\SS_n$ if $p<q_n$\,,
    and in $\SS_n\setminus\{\unit\}$ if $p=q_n$\,.
  \item\label{item:5} The initial measure of the chain coincides
    with~$h$, which is a probability distribution on~$\SS_n$\,. The
    initial distribution puts positive weight on every non unit simple
    braid, and on the unit $\unit$ if and only if $p<q_n$\,.
\item\label{item:9} The transition matrix $P$ of the chain $(X_k)_{k\geq1}$
    is the following:
  \begin{align}
\label{eq:18}
P_{x,x'}&=\un{x\to x'}p^{\abs{x}}\frac{h(x')}{h(x)},
  \end{align}
  where $x$ and $x'$ range over $\SS_n$ for $p<q_n$ or over
  $\SS_n\setminus\{\unit\}$ for $p=q_n$.
\end{enumerate}
\end{theorem}

\begin{proof}
Let $f:\bbr n\to\bbR$ be defined by $f(x)=p^{\abs{x}}$\,.

We first show that $h>0$ on~$\SS_n$ if $p<q_n$\,, and that $h>0$ on
$\SS_n\setminus\{\unit\}$ if $p=q_n$\,. Obviously, it follows from
Proposition~\ref{cor:1} that $h$ is non-negative on~$\SS_n$\,, and even
on~$\bbr n$\,.
\begin{enumerate}
\item \emph{Case $p<q_n$.} Then $H_n(p)>0$ and therefore, according to
    Theorem~\ref{thr:1} and Proposition~\ref{cor:1}, we obtain:
\[h(x) = \nu_p(\CC_x) \geq \nu_p(\{x\}) = H_n(p) \cdot p^{\abs{x}} > 0 \text{ for all } x \in \SS_n,\]
which was to be shown.
\item \emph{Case $p=q_n$}\,.  Consider the matrix $B$ indexed by pairs
  $(x,x')$ of simple braids distinct from $\unit$ and from~$\Delta$\,,
  and defined by $B_{x,x'}=\un{x\to x'}q_n^{\abs{x'}}$\,. According to
  Lemma~\ref{lem:3}, the restriction of $g$ to
  $\SS_n\setminus\{\unit,\Delta\}$ is the Perron right eigenvector
  of~$B$, where $g$ has been defined in~(\ref{eq:17}). Therefore $g>0$
  on $\SS_n\setminus\{\unit,\Delta\}$. But $h(x)=q_n^{\abs{x}}g(x)$
  holds on $\SS_n$ according to Lemma~\ref{lem:2}, therefore $h>0$ on
  $\SS_n\setminus\{\unit,\Delta\}$.  As for~$\Delta$\,, one has
  $h(\Delta)=p^{\abs{\Delta}}>0$. Hence $h>0$ on
  $\SS_n\setminus\{\unit\}$, as claimed.
\end{enumerate}

It follows in particular from the above discussion that the matrix $P$
defined in the statement of the theorem is well defined. Now, let
$(x_1,\ldots,x_k)$ be any sequence of simple braids (including maybe the unit
braid). Let $\delta$ and $\delta'$ denote the following quantities:
\begin{align*}
  \delta&=\nu_p(X_1=x_1,\ldots,X_k=x_k)\\
  \delta'&=h(x_1)\cdot P_{x_1,x_2}\cdot\ldots\cdot P_{x_{k-1},x_k}\,.
\end{align*}
We prove that $\delta=\delta'$\,.

We observe first that both $\delta$ and $\delta'$ are zero if the sequence
$(x_1,\ldots,x_k)$ is not normal. Hence, without loss of generality, we
restrict our analysis to the case where $(x_1,\ldots,x_k)$ is a normal
sequence of simple braids.

Consider the braid $y=x_1\cdot\ldots\cdot x_k$\,.  By the uniqueness of the
Garside normal form, the following equality holds:
\[
  \{X_1=x_1,\ldots,X_k=x_k\}=\{X_1 \cdot\ldots\cdot X_k =y\}\,.
\]
Applying successively Proposition~\ref{cor:1} and Corollary~\ref{cor:cor}, we
have thus:
\[
  \delta=h(y)
%   =\sum_{z\in \DD_n(x_k)}(-1)^{\varphi(z)}f(y\cdot z)%=f(y) \cdot \sum_{z\in \DD_n(x_k)}(-1)^{\varphi(z)}f(z)
  =p^{\abs{x_1}+\ldots+\abs{x_{k-1}}}h(x_k)\,.\]
On the other hand, we have:
\[  \delta'=h(x_1)\cdot p^{\abs{x_1}}\frac{h(x_2)}{h(x_1)}\cdot\ldots\cdot
  p^{\abs{x_{k-1}}}\frac{h(x_k)}{h(x_{k-1})}=p^{\abs{x_1}+\ldots+\abs{x_{k-1}}}h(x_k)
\]
which completes the proof of the equality $\delta=\delta'$. It follows that
$(X_k)_{k\geq1}$ is indeed a Markov chain with the specified initial
distribution and transition matrix.

If $p=q_n$\,, then we have already observed in the proof of
Lemma~\ref{lem:3} that $h(\unit)=0$. It implies both
$\nu(X_1=\unit)=0$ and $P_{x,\unit}=0$ for all
$x\in\SS_n\setminus\{\unit\}$. We conclude that $(X_k)_{k\geq1}$ does
never reach~$\unit$, which completes the proof of point~\ref{item:8},
and of the theorem.
\end{proof}

\paragraph*{Running examples for $n=3$.}

We characterize the uniform measure at infinity both for $\br 3$ and
for~$\dbr 3$\,. For this, we first determine the root of smallest modulus of
the Möbius polynomial, which we determined in
Subsection~\ref{sec:growth-series-braid}: $q_3=(\sqrt5-1)/2$ for $\br 3$ and
$q_3=1/2$ for~$\dbr 3$\,.

The Markov chain of simple braids induced by the uniform measure at infinity
takes its values in $\SS_3\setminus\{\unit\}$, which has $5$ elements for
$\br 3$ and $4$ elements for~$\dbr 3$\,.  Since the Möbius transform of
$q_3^{|x|}$ is tabulated in Table~\ref{tab:mobiustransform3}, we are in the
position to compute both the initial distribution and the transition matrix
of the chain by an application of Theorem~\ref{thr:2}, yielding the results
given in Table~\ref{tab:unfimre}.

\begin{table}
  \begin{gather*}
    \begin{array}{c}
      \Delta_3\\\sigma_1\\\sigma_2\\\sigma_1\cdot\sigma_2\\\sigma_2\cdot\sigma_1
    \end{array}
    \begin{pmatrix}
\sqrt5-2&\sqrt5-2&\sqrt5-2&(7-3\sqrt5)/2&(7-3\sqrt5)/2\\
0&(\sqrt5-1)/2&0&(3-\sqrt5)/2&0\\
0&0&(\sqrt5-1)/2&0&(3-\sqrt5)/2\\
0&0&(\sqrt5-1)/2&0&(3-\sqrt5)/2\\
0&(\sqrt5-1)/2&0&(3-\sqrt5)/2&0
    \end{pmatrix}
\\[1em]
    \begin{array}{c}
\delta_3\\\sigma_{1,2}\\\sigma_{2,3}\\\sigma_{1,3}
    \end{array}
    \begin{pmatrix}
1/4&1/4&1/4&1/4\\
0&1/2&0&1/2\\
0&1/2&1/2&0\\
0&0&1/2&1/2
    \end{pmatrix}
  \end{gather*}
  \caption{Transition matrix of the Markov
    chain on simple braids induced by the uniform measure at infinity
    for $\br 3$ (up) and for $\dbr 3$ (down). The initial distribution
    of the chain can be read on the first line of each matrix.}
  \label{tab:unfimre}
\end{table}

\section{Applications to finite uniform distributions}
\label{sec:appl-asympt-finite}

\subsection{Weak convergence of finite uniform distributions}
\label{sec:cons-furth-quest}

The following result states a relationship between the finite uniform
distributions, and the uniform measure at infinity. If one were only
interested in finite uniform distributions, that would be a justification for
studying uniform measures as defined previously.

\begin{theorem}
  \label{thr:9}
  The uniform measure at infinity $\nu_{q_n}$ is the weak limit of the
  sequence~$(\mu_{n,k})_{k\geq0}$ as $k\to\infty$\,, where $\mu_{n,k}$
  is for each integer $k\geq0$ the uniform distribution on the finite
  set $\bbr n(k)$ defined by:
\[
  \bbr n(k)=\{x\in\bbr n\tq\abs{x}=k\}\,.
\]
\end{theorem}

\begin{proof}
  Recall that $\bbr n(k)$, as a subset of $\bbr n$, is identified with
  its image in~$\bbrbar n$\,, and thus $\mu_{n,k}$ identifies with a
  discrete probability distribution on~$\bbrbar n$\,. We denote
  $\lambda_n(k)=\#\bbr n(k)$.  For a fixed braid $x\in\bbr n$\,, the
  map $y\mapsto x\cdot y$ is a bijection between $\bbr n$ and
  $\Up x\cap\bbr n$, a fact that we already used in the proof of
  Theorem~\ref{thr:1}, point~\ref{item:7}.  Hence, for any
  $k\geq |x|$, and using the asymptotics found in
  Corollary~\ref{cor:6}, one has:
\begin{align*}
  \mu_{n,k}(\Up x)&=\frac{\lambda_n(k-\abs{x})}{\lambda_n(k)}\to_{k\to\infty}q_n^{\abs{x}}\,.
\end{align*}

Invoking the Portemanteau theorem as in the proof of Theorem~\ref{thr:1}, we
deduce that any weak cluster value $\nu$ of $(\mu_{n,k})_{k\geq0}$ satisfies
$\nu(\Up x)=q_n^{\abs{x}}$ for any full visual cylinder~$\Up x$\,.
Theorem~\ref{thr:1} implies $\nu=\nu_{q_n}$\,. By compactness of~$\bbrbar
n$\,, it follows that $(\mu_{n,k})_{k\geq0}$ converges toward~$\nu_{q_n}$.
\end{proof}

A practical interest of Theorem~\ref{thr:9} lies in the following corollary.
Define $X_i:\bbrbar n\to\SS_n$ by $X_i(\xi)=x_i$\,, where $(x_k)_{k\geq1}$ is
the extended Garside normal form of~$\xi$.

\begin{corollary}
\label{cor:4} Let $j\geq1$ be an integer.  Then the joint distribution of the first
$j$ simple braids appearing in the extended Garside decomposition of a
uniformly random braid of size~$k$ converges, as $k\to\infty$, toward the
joint distribution of\/ $(X_1,\ldots,X_j)$ under the uniform measure at infinity.
\end{corollary}

\begin{proof}
By definition of the topology on $\bbrbar n$, the mapping $\xi\in\bbrbar
n\mapsto (X_1,\ldots,X_j)$ is continuous for each integer $j\geq1$. The
result follows thus from Theorem~\ref{thr:9}.
\end{proof}

\paragraph*{Example for $n=4$.}
\label{sec:running-example-n=3}

In anticipation of the computations to be performed in
Section~\ref{se-explicit}, we depict in Figure~\ref{fig:ranazdopsq} the
beginning of a ``truly random infinite braid'' on $n=4$ strands, up to height
$k=10$. These are ``typical 10 first elements'' in the decomposition of a
large random braid on four strands. Observe the absence of~$\Delta$; the
numerical values found in next subsection make it quite likely.

\begin{figure}
\centering  \begin{tikzpicture}
\draw[thick,green] (0,0) -- (2.5,0);	% first strand (green)
\draw[thick,green] (2.5,0) -- (3,0.5);
\draw[thick,green] (3,0.5) -- (3.25,0.5);
\draw[thick,green] (3.25,0.5) -- (3.45,0.3);
\draw[thick,green] (3.55,.2) -- (3.75,0);
\draw[thick,green] (3.75,0) -- (7.75,0);
\draw[thick,green] (7.75,0) -- (8.25,0.5);
\draw[thick,green] (8.25,.5) -- (10,.5);
\draw[thick,green] (10,.5) -- (10.5,1);
\draw[thick,green] (10.5,1) -- (10.75,1);
\draw[thick,green] (10.75,1) -- (10.95,.8);
\draw[thick,green] (11.05,.7) -- (11.25,.5);
\draw[thick,green] (11.25,.5) -- (11.5,.5);
\draw[thick,green] (11.5,.5) -- (11.7,.3);
\draw[thick,green] (11.8,.2) -- (12,0);
\draw[thick,green] (12,0) -- (12.25,0);
\draw[thick,green,dashed] (12,0) -- (15,0);
\draw[thick,red] (0,.5) -- (1.75,.5);	% second strand (red)
\draw[thick,red] (1.75,.5) -- (2.25,1);
\draw[thick,red] (2.25,1) -- (4,1);
\draw[thick,red] (4,1) -- (4.2,.8);
\draw[thick,red] (4.3,.7) -- (4.5,.5);
\draw[thick,red] (4.5,.5) -- (4.75,.5);
\draw[thick,red] (4.75,.5) -- (5.25,1);
\draw[thick,red] (5.25,1) -- (5.5,1);
\draw[thick,red] (5.5,1) -- (5.7,.8);
\draw[thick,red] (5.8,.7) -- (6,.5);
\draw[thick,red] (6,.5) -- (6.25,.5);
\draw[thick,red] (6.25,.5) -- (6.75,1);
\draw[thick,red] (6.75,1) -- (7,1);
\draw[thick,red] (7,1) -- (7.2,.8);
\draw[thick,red] (7.3,.7) -- (7.5,.5);
\draw[thick,red] (7.5,.5) -- (7.75,.5);
\draw[thick,red] (7.75,.5) -- (7.95,.3);
\draw[thick,red] (8.05,.2) -- (8.25,0);
\draw[thick,red] (8.25,0) -- (11.5,0);
\draw[thick,red] (11.5,0) -- (12,0.5);
\draw[thick,red] (12,.5) -- (13.75,.5);
\draw[thick,red] (13.75,.5) -- (14.25,1);
\draw[thick,red] (14.25,1) -- (14.5,1);
\draw[thick,red,dashed] (14.5,1) -- (15,1);
\draw[thick,blue] (0,1) -- (0.25,1);	% third strand (blue)
\draw[thick,blue] (0.25,1) -- (0.75,1.5);
\draw[thick,blue] (0.75,1.5) -- (1,1.5);
\draw[thick,blue] (1,1.5) -- (1.2,1.3);
\draw[thick,blue] (1.3,1.2) -- (1.5,1);
\draw[thick,blue] (1.5,1) -- (1.75,1);
\draw[thick,blue] (1.75,1) -- (1.95,.8);
\draw[thick,blue] (2.05,.7) -- (2.25,.5);
\draw[thick,blue] (2.25,.5) -- (2.5,.5);
\draw[thick,blue] (2.5,.5) -- (2.7,.3);
\draw[thick,blue] (2.8,.2) -- (3,0);
\draw[thick,blue] (3,0) -- (3.25,0);
\draw[thick,blue] (3.25,0) -- (3.75,0.5);
\draw[thick,blue] (3.75,.5) -- (4,.5);
\draw[thick,blue] (4,.5) -- (4.5,1);
\draw[thick,blue] (4.5,1) -- (4.75,1);
\draw[thick,blue] (4.75,1) -- (4.95,.8);
\draw[thick,blue] (5.05,.7) -- (5.25,.5);
\draw[thick,blue] (5.25,.5) -- (5.5,.5);
\draw[thick,blue] (5.5,.5) -- (6,1);
\draw[thick,blue] (6,1) -- (6.25,1);
\draw[thick,blue] (6.25,1) -- (6.45,.8);
\draw[thick,blue] (6.55,.7) -- (6.75,.5);
\draw[thick,blue] (6.75,.5) -- (7,.5);
\draw[thick,blue] (7,.5) -- (7.5,1);
\draw[thick,blue] (7.5,1) -- (8.5,1);
\draw[thick,blue] (8.5,1) -- (9,1.5);
\draw[thick,blue] (9,1.5) -- (9.25,1.5);
\draw[thick,blue] (9.25,1.5) -- (9.45,1.3);
\draw[thick,blue] (9.55,1.2) -- (9.75,1);
\draw[thick,blue] (9.75,1) -- (10,1);
\draw[thick,blue] (10,1) -- (10.2,.8);
\draw[thick,blue] (10.3,.7) -- (10.5,.5);
\draw[thick,blue] (10.5,.5) -- (10.75,.5);
\draw[thick,blue] (10.75,.5) -- (11.25,1);
\draw[thick,blue] (11.25,1) -- (12.25,1);
\draw[thick,blue] (12.25,1) -- (12.75,1.5);
\draw[thick,blue] (12.75,1.5) -- (13,1.5);
\draw[thick,blue] (13,1.5) -- (13.2,1.3);
\draw[thick,blue] (13.3,1.2) -- (13.5,1);
\draw[thick,blue] (13.5,1) -- (13.75,1);
\draw[thick,blue] (13.75,1) -- (13.95,.8);
\draw[thick,blue] (14.05,.7) -- (14.25,.5);
\draw[thick,blue] (14.25,.5) -- (14.5,.5);
\draw[thick,blue,dashed] (14.5,.5) -- (15,.5);
\draw[thick,orange] (0,1.5) -- (0.25,1.5);	% fourth strand (black)
\draw[thick,orange] (0.25,1.5) -- (0.45,1.3);
\draw[thick,orange] (0.55,1.2) -- (0.75,1);
\draw[thick,orange] (0.75,1) -- (1,1);
\draw[thick,orange] (1,1) -- (1.5,1.5);
\draw[thick,orange] (1.5,1.5) -- (8.5,1.5);
\draw[thick,orange] (8.5,1.5) -- (8.7,1.3);
\draw[thick,orange] (8.8,1.2) -- (9,1);
\draw[thick,orange] (9,1) -- (9.25,1);
\draw[thick,orange] (9.25,1) -- (9.75,1.5);
\draw[thick,orange] (9.75,1.5) -- (12.25,1.5);
\draw[thick,orange] (12.25,1.5) -- (12.45,1.3);
\draw[thick,orange] (12.55,1.2) -- (12.75,1);
\draw[thick,orange] (12.75,1) -- (13,1);
\draw[thick,orange] (13,1) -- (13.5,1.5);
\draw[thick,orange] (13.5,1.5) -- (13.75,1.5);
\draw[thick,orange,dashed] (13.75,1.5) -- (15,1.5);
\node at (-.5,0){$1$};
\node at (-.5,.5){$2$};
\node at (-.5,1){$3$};
\node at (-.5,1.5){$4$};
\node at (0.5,-.5){$\sigma_3$};
\node at (1.25,-.5){$\sigma_3$};
\node at (2,-.5){$\sigma_2$};
\node at (2.75,-.5){$\sigma_1$};
\node at (3.5,-.5){$\sigma_1$};
\node at (4.25,-.5){$\sigma_2$};
\node at (5.0,-.5){$\sigma_2$};
\node at (5.75,-.5){$\sigma_2$};
\node at (6.5,-.5){$\sigma_2$};
\node at (7.25,-.5){$\sigma_2$};
\node at (8.0,-.5){$\sigma_1$};
\node at (8.75,-.5){$\sigma_3$};
\node at (9.5,-.5){$\sigma_3$};
\node at (10.25,-.5){$\sigma_2$};
\node at (11.0,-.5){$\sigma_2$};
\node at (11.75,-.5){$\sigma_1$};
\node at (12.5,-.5){$\sigma_3$};
\node at (13.25,-.5){$\sigma_3$};
\node at (14.0,-.5){$\sigma_2$};
\end{tikzpicture}
\caption{A random uniform braid on four strands. The first elements of
  its Garside decomposition are:
  $(\sigma_3) \cdot(\sigma_3\sigma_2\sigma_1) \cdot(\sigma_1\sigma_2)
  \cdot(\sigma_2) \cdot(\sigma_2) \cdot(\sigma_2)
  \cdot(\sigma_2\sigma_1\sigma_3) \cdot(\sigma_3\sigma_2)
  \cdot(\sigma_2\sigma_1\sigma_3) \cdot(\sigma_3\sigma_2) $}
  \label{fig:ranazdopsq}
\end{figure}
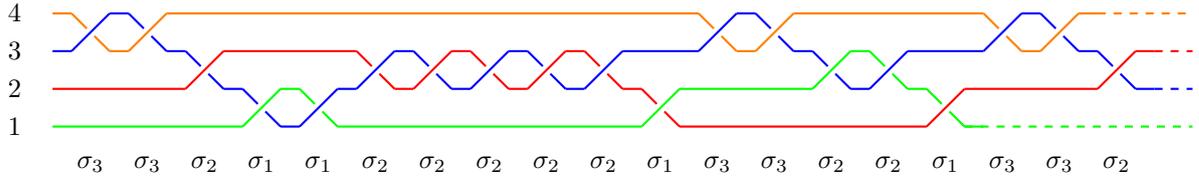

\subsection{A geometric number of $\Delta$s.}
\label{sec:geom-numb-delt}

The $\Delta$ element only appears at the beginning of normal sequences
of simple braids. Accordingly, under the uniform probability
measure~$\nu_p$\,, the occurrences of $\Delta$ in the Markov chain
$(X_k)_{k\geq1}$ are only observed in the first indices, if any, and
their number is geometrically distributed.

This behavior is quite easy to quantify, as the probabilistic parameters
associated with $\Delta$ have simple expressions:
\begin{align*}
  \nu_p(X_1=\Delta)&=h(\Delta)=p^{\abs{\Delta}}=
  \begin{cases}
   p^{\frac{n(n-1)}2} & \text{ if } \bbr n = \br n \\
   p^{n-1} & \text{ if } \bbr n = \dbr n
  \end{cases}
\,,&
P_{\Delta,\Delta}&=p^{\abs{\Delta}}\,.
\end{align*}

It follows that the number of $\Delta$s appearing in the normal form of a
random braid, possibly infinite and distributed according to a uniform measure of
parameter $p\in(0,q_n]$\,, is geometric of parameter~$p^{\abs{\Delta}}$\,.

As a consequence of Theorem~\ref{thr:9}, we obtain this corollary.

\begin{corollary}
  \label{cor:2}
  Let $T_k:\bbr n(k)\to\bbN$ denote the number of $\Delta$s in the
  Garside decomposition of a random braid of size~$k$. Then
  $(T_k)_{k\geq0}$ converges in distribution, as $k\to\infty$\,,
  toward a geometric distribution of parameter~$q_n^{\frac{n(n-1)}2}$ if $\bbr n = \br n$,
  or~$q_n^{n-1}$ if $\bbr n = \dbr n$.
\end{corollary}

Authors are sometimes only interested by the elements of the Garside
decomposition of a ``large'' braid that appear \emph{after} the last
occurrence of~$\Delta$. The notion of uniform measure at infinity allows also
to derive information on these, as we shall see next.

\paragraph{Examples for $n=3$ and $n=4$.}

Exact or numerical values for the parameter of the geometrical distribution are easily
computed for $n=3$ and for $n=4$, based on our previous computations for
$n=3$ and on the computations of Section~\ref{se-explicit} for $n=4$; see
Table~\ref{tab:pjkaozdja}.

\begin{table}
\begin{gather*}
  \begin{array}{l|cc|cc|}
\multicolumn{1}{l}{}&\multicolumn{2}{c}{\text{monoid $\br
  n$}}&\multicolumn{2}{c}{\text{monoid $\dbr n$}}\\
&\text{parameter}&\text{prob. of occ.~of $\Delta$}
&\text{parameter}&\text{prob. of occ.~of $\Delta$}\\
\cline{2-5}
n=3&\Bigl(-\frac12+\frac{\sqrt5}2\Bigr)^3\approx0.236&\approx.309&
\frac14&\frac13\\
n=4&\approx0.0121 &\approx 0.0122&\Bigl(\frac12-\frac{\sqrt5}{10}\Bigr)^3\approx0.021&
\approx0.022
  \end{array}
\end{gather*}
\caption{For each monoid, we tabulate on the left hand side the
  parameter of the geometric distribution of the number of~$\Delta$s appearing
  in a random infinite braid. On the right hand side, we tabulate the
  associated probability of occurrence of at least one~$\Delta$.}
\label{tab:pjkaozdja}
\end{table}

\subsection{On a conjecture by Gebhardt and Tawn}
\label{sec:gebhardt--conjecture}

This subsection only deals with positive braids monoids~$\br n$\,.  In their
\emph{Stable region Conjecture}, the authors of~\cite{gebhardt14} suggest the
following, based on a thorough experimental analysis. For each integer
$i\geq1$, let $\lambda_{i*}(\mu_k)$ be the distribution of
the~$i^{\text{th}}$~factor in the extended Garside normal that occurs after
the last~$\Delta$\,, for braids drawn at random uniformly among braids of
length~$k$. Then two facts are suspected to hold, according
to~\cite[Conjecture~3.1]{gebhardt14}:
\begin{description}
\item[\normalfont\itshape Presumed Fact 1.] For each integer $i\geq1$, the
    sequence $(\lambda_{i*}(\mu_k))_{k\geq1}$ is convergent as
    $k\to\infty$\,.
\item[\normalfont\itshape Presumed Fact 2.] There exists a probability
    measure, say $\mu$ on~$\SS_n\setminus\{\unit,\Delta\}$\,, and a
    constant $C>0$ such that holds:
  \begin{gather}
    \label{eq:6}
\forall i>C\quad \lambda_{i*}(\mu_k)\to\mu\quad\text{as $k\to\infty$\,.}
  \end{gather}
\end{description}

Theorems~\ref{thr:9} and~\ref{thr:2} translate the problem of the limiting
behavior of factors of the normal form within the familiar field of Markov
chains with a finite number of states. This brings a simple way of
determining the status of the above conjecture.

It follows from Theorem~\ref{thr:9} that the distribution of the
$k^{\text{th}}$ element in the extended Garside decomposition of a
random braid (including all the starting~$\Delta$s), distributed
uniformly among finite braids of size~$k$, converges toward the
distribution of the $k^{\text{th}}$ element in the extended Garside
decomposition of an \emph{infinite} braid, distributed according to
the unique uniform measure at infinity. And, according to
Theorem~\ref{thr:2}, this is the distribution of a Markov chain at
time~$k$, with the prescribed initial distribution and transition
matrix.

As for $\lambda_{i*}(\mu_k)$ it converges thus toward the distribution of the same
chain, $i$~steps after having left the state~$\Delta$\,. Hence we can affirm
\emph{the veracity of Fact~1}. Using the notations of Theorem~\ref{thr:2}, we
may also describe the limit, say
$\lambda_{i*}=\lim_{k\to\infty}\lambda_{i*}(\mu_k)$\,, as follows:
\[
\forall s\in\SS_n\setminus\{\unit,\Delta\}\quad
\lambda_{1*}(s)=\frac{h(s)}{1-{q_n}^{\frac{n(n-1)}2}}\,,
\]
where the denominator comes from the conditioning on $s\neq\Delta$\,. The
next values for $\lambda_{i*}$ are obtained recursively:
\begin{gather}
  \label{eq:25}
  \forall i\geq 1\quad \lambda_{i*}=\lambda_{1*}P^{i-1}\,,
\end{gather}
where $P$ is the transition matrix of the chain described in
Theorem~\ref{thr:2}.

On the contrary, \emph{Fact~2 is incorrect.} Indeed, keeping the
notation $\lambda_{i*}=\lim_{k\to\infty}\lambda_{i*}(\mu_k)$, if~\eqref{eq:6}
was true, then $\mu=\lambda_{i*}$ for $i$
large enough, would be the invariant measure of the Markov chain
according to~\eqref{eq:25}. But that would imply that the chain is
stationary. We prove below that this is not the case when $n> 3$. What
is true however, is that $(\lambda_{i*})_{i\geq1}$ converges toward
the stationary measure of the chain when $i\to\infty$\,.

Assume, seeking a contradiction, that the chain $(X_k)_{k\geq1}$ is
stationary. That would imply that the Möbius transform $h$ of the function
$f(x)=q_n^{\abs{x}}$\,, identified with a vector indexed by
$\SS_n\setminus\{\unit,\Delta\}$\,, is left invariant for the transition
matrix~$P$. Hence, for $y\in\SS_n\setminus\{\unit,\Delta\}$:
\[
  h(y)=(hP)_y=\sum_{x\in\SS_n\setminus\{\unit,\Delta\}\tq x\to y}h(x)\frac{f(x)h(y)}{h(x)}\,.
\]
We deduce, since $h>0$ on $\SS_n\setminus\{\unit\}$\,:
\begin{gather}
\label{eq:32}
  \forall y\in\SS_n\setminus\{\unit,\Delta\}\quad
\sum_{x\in\SS_n\setminus\{\unit,\Delta\}\tq x\to y}q_n^{\abs{x}}=1\,.
\end{gather}

Consider $y=\sigma_1$ and $y'=\sigma_1\cdot\sigma_2\cdot\sigma_1$\,. One has:
\[
  \Dl(y)=\{\sigma_1\}\subsetneq\Dl(y')=\{\sigma_1,\sigma_2\}\,,
\]
which entails:
\[
  \{x\in\SS_n\tq (x\neq\unit,\,\Delta)\wedge x\to y\}\supsetneq
    \{x\in\SS_n\tq (x\neq\unit,\,\Delta)\wedge x\to y'\}\,.
\]
It follows that the equality stated in~\eqref{eq:32} cannot hold both for $y$
and for~$y'$, which is the sought contradiction.

\section{Explicit computations}\label{se-explicit}

We gather in this section the computations needed, for $n=4$, to characterize
the uniform measure at infinity, both for $\br 4$ and for~$\dbr 4$.

\subsection{Computations for $\br 4$}
\label{sec:computations-br-4}

The monoid $\br 4$ has the following presentation:
\begin{gather*}
\br 4=\bigl\langle\sigma_1,\; \sigma_2,\;\sigma_3
\ \big|\
\sigma_1\sigma_3=\sigma_3\sigma_1,\
\sigma_1\sigma_2\sigma_1=\sigma_2\sigma_1\sigma_1,\
\sigma_2\sigma_3\sigma_2=\sigma_3\sigma_2\sigma_3\bigr\rangle^+\,.
\end{gather*}
In order to shorten notations, we denote a product of generators $\sigma_i$
simply by the corresponding sequence of indices. So for instance, the Garside
element is denoted: $\Delta_4=123121$\,.

The lattice $\DD_4=\{\Delta_X\ |\ X\subseteq\Sigma_4\}$ has $2^3=8$ elements,
and is isomorphic to the lattice of subsets of~$\{1,2,3\}$, whereas $\SS_4$
has $4!=24$ elements. The Hasse diagram of $\SS_4$ is depicted in
Figure~\ref{fig:hessebr4}.

\begin{figure}
\centering
%\begin{center}
\begin{tikzpicture}
\node at (0,0) {$\unit$};

\node at (-2.5,1) {$\sigma_1$}; \node at (0,1) {$\sigma_2$}; \node at (2.5,1)
{$\sigma_3$};

\node at (-5,2) {$\sigma_1 \CDOT \sigma_2$}; \node at (-2.5,2) {$\sigma_2
\CDOT \sigma_1$}; \node at (0,2) {$\sigma_1 \CDOT \sigma_3$}; \node at
(2.5,2) {$\sigma_2 \CDOT \sigma_3$}; \node at (5,2) {$\sigma_3 \CDOT
\sigma_2$};

\node at (-6.25,3) {$\sigma_1 \CDOT \sigma_2 \CDOT \sigma_3$}; \node at
(-3.75,3) {$\sigma_1 \CDOT \sigma_2 \CDOT \sigma_1$}; \node at (-1.25,3)
{$\sigma_1 \CDOT \sigma_3 \CDOT \sigma_2$}; \node at (1.25,3) {$\sigma_2
\CDOT \sigma_1 \CDOT \sigma_3$}; \node at (3.75,3) {$\sigma_2 \CDOT \sigma_3
\CDOT \sigma_2$}; \node at (6.25,3) {$\sigma_3 \CDOT \sigma_2 \CDOT
\sigma_1$};

\node at (-5,4) {$\sigma_1 \CDOT \sigma_2 \CDOT \sigma_1 \CDOT \sigma_3$};
\node at (-2.5,4) {$\sigma_1 \CDOT \sigma_2 \CDOT \sigma_3 \CDOT \sigma_2$};
\node at (0,4) {$\sigma_2 \CDOT \sigma_1 \CDOT \sigma_3 \CDOT \sigma_2$};
\node at (2.5,4) {$\sigma_1 \CDOT \sigma_3 \CDOT \sigma_2 \CDOT \sigma_1$};
\node at (5,4) {$\sigma_2 \CDOT \sigma_3 \CDOT \sigma_2 \CDOT \sigma_1$};

\node at (-2.5,5) {$\sigma_1 \CDOT \sigma_2 \CDOT \sigma_1 \CDOT \sigma_3
\CDOT \sigma_2$}; \node at (0,5) {$\sigma_1 \CDOT \sigma_2 \CDOT \sigma_3
\CDOT \sigma_2 \CDOT \sigma_1$}; \node at (2.5,5) {$\sigma_2 \CDOT \sigma_1
\CDOT \sigma_3 \CDOT \sigma_2 \CDOT \sigma_1$};

\node at (0,6) {$\Delta_4$};

\draw (0.25,0.1) -- (2.25,0.9); \draw (0,0.25) -- (0,0.8); \draw (-0.25,0.1)
-- (-2.25,0.9);

\draw (2.75,1.1) -- (4.5,1.8); \draw (2.25,1.1) -- (0.5,1.8); \draw
(0.25,1.1) -- (2,1.8); \draw (-0.25,1.1) -- (-2,1.8); \draw (-2.75,1.1) --
(-4.5,1.8); \draw (-2.25,1.1) -- (-0.5,1.8);

\draw (5.25,2.2) -- (6,2.8); \draw (4.75,2.2) -- (4,2.8); \draw (2.75,2.2) --
(3.5,2.8); \draw (2.25,2.2) -- (1.5,2.8); \draw (-0.25,2.2) -- (-1,2.8);
\draw (-2.75,2.2) -- (-3.5,2.8); \draw (-2.05,2.12) -- (0.6875,2.85); \draw
(-5.25,2.2) -- (-6,2.8); \draw (-4.75,2.2) -- (-4,2.8);

\draw (6,3.2) -- (5.25,3.8); \draw (5.6875,3.15) -- (3.25,3.8); \draw (4,3.2)
-- (4.75,3.8); \draw (1,3.2) -- (0.25,3.8); \draw (-0.6875,3.15) --
(1.75,3.8); \draw (-1.5,3.2) -- (-2.25,3.8); \draw (-4,3.2) -- (-4.75,3.8);
\draw (-6,3.2) -- (-5.25,3.8); \draw (-5.6875,3.15) -- (-3.25,3.8);

\draw (4.25,4.2) -- (2.75,4.8); \draw (1.75,4.2) -- (0.25,4.8); \draw
(0.25,4.2) -- (1.75,4.8); \draw (-0.25,4.2) -- (-1.75,4.8); \draw (-1.75,4.2)
-- (-0.25,4.8); \draw (-4.25,4.2) -- (-2.75,4.8);

\draw (1.75,5.2) -- (0.25,5.8); \draw (0,5.2) -- (0,5.8); \draw (-1.75,5.2)
-- (-0.25,5.8);

\draw (0,0) circle (0.25); \draw (-2.45,1.2) arc (90:-90:0.2) -- (-2.55,0.8)
arc (270:90:0.2) -- cycle; \draw (0.05,1.2) arc (90:-90:0.2) -- (-0.05,0.8)
arc (270:90:0.2) -- cycle; \draw (2.55,1.2) arc (90:-90:0.2) -- (2.45,0.8)
arc (270:90:0.2) -- cycle; \draw (0.45,2.2) arc (90:-90:0.2) -- (-0.45,1.8)
arc (270:90:0.2) -- cycle; \draw (-3.2,3.2) arc (90:-90:0.2) -- (-4.3,2.8)
arc (270:90:0.2) -- cycle; \draw (4.3,3.2) arc (90:-90:0.2) -- (3.2,2.8) arc
(270:90:0.2) -- cycle; \draw (0.2,6.2) arc (90:-90:0.2) -- (-0.2,5.8) arc
(270:90:0.2) -- cycle;
\end{tikzpicture}
%\end{center}
\caption{Hasse diagram of $\SS_4$ for $\br 4$\,. Elements of $\DD_4$ are
circled} \label{fig:hessebr4}
\end{figure}

In order to compute the Möbius transform $h$ of the function $f(x)=p^{|x|}$
defined on~$\SS_4$, we refer to the expression~(\ref{eq:3}):
\begin{gather*}
  h(x)=\sum_{X\subseteq\Sigma\tq x\cdot\Delta_X\in\SS_4}(-1)^{|X|}f(x\cdot\Delta_X)
\end{gather*}

Furthermore, recalling the property $x\cdot\Delta_X\in\SS_4\iff
X\subseteq\Dl(\Delta_{\Sigma\setminus\Dr(x)})$ proved earlier
in~(\ref{eq:24}), the range of those $X\subseteq\Sigma$ such that
$x\cdot\Delta_X\in\SS_4$ is directly derived from the knowledge of the sets
$\Dl(y)$ and~$\Dr(y)$. All these elements are gathered in
Table~\ref{tab:mobiusB4}.

\begin{table}
  \[
  \begin{array}{rclll}
    \Dl(x)&x\in\SS_4&\Dr(x)&\{y\in\SS_4\tq x\to y\}&h(x)\\
\hline
\rule{0em}{1em}\emptyset&\fbox{$\unit$}&\emptyset&\unit&1-3p+p^2+2p^3-p^6\\
1&\fbox1&1&1,12,123&p-2p^2+p^4\\
2&\fbox2&2&2,21,23,213,2132&p-2p^2+p^3\\
3&\fbox3&3&3,32,321&p-2p^2+p^4\\
1&12&2&2,21,23,213,2132&p^2-2p^3+p^4\\
3,1&\fbox{13}&1,3&1,3,12,13,32,123,132,321,1232,1321,12321&p^2-p^3\\
2&21&1&1,12,123&p^2-2p^3+p^5\\
2&23&3&3,32,321&p^2-2p^3+p^5\\
3&32&2&2,21,23,213,2132&p^2-2p^3+p^4\\
2,1&\fbox{121}&1,2&1,2,12,21,23,121,123,213,1213,2132,21323&p^3-p^4\\
1&123&3&3,32,321&p^3-2p^4+p^6\\
3,1&132&2&2,21,23,213,2132&p^3-2p^4+p^5\\
2&213&1,3&1,3,12,13,32,123,132,321,1232,1321,12321&p^3-p^4\\
3,2&\fbox{232}&2,3&2,3,21,23,32,213,232,321,2132,2321,21321&p^3-p^4\\
3&321&1&1,12,123&p^3-2p^4+p^6\\
2,1&1213&1,3&1,3,12,13,32,123,132,321,1232,1321,12321&p^4-p^5\\
3,1&1232&2,3&2,3,21,23,32,213,232,321,2132,2321,21321&p^4-p^5\\
3,1&1321&1,2&1,2,12,21,23,121,123,213,1213,2132,21323&p^4-p^5\\
2&2132&2&2,21,23,213,2132&p^4-2p^5+p^6\\
3,2&2321&1,3&1,3,12,13,32,123,132,321,1232,1321,12321&p^4-p^5\\
3,1&12321&1,3&1,3,12,13,32,123,132,321,1232,1321,12321&p^5-p^6\\
3,2&21321&1,2&1,2,12,21,23,121,123,213,1213,2132,21323&p^5-p^6\\
2,1&21323&2,3&2,3,21,23,32,213,232,321,2132,2321,21321&p^5-p^6\\
3,2,1&\fbox{$\Delta_4$}&1,2,3&\SS_4&p^6
  \end{array}
\]
\caption{Characteristic elements for~$\br 4$\,. Elements of $\DD_4$
  are framed. The column $h(x)$
  tabulates the Möbius transform of the function $f(x)=p^{|x|}$ on~$\SS_4$\,.}
  \label{tab:mobiusB4}
\end{table}

The Möbius polynomial $H_4(t)$ can be obtained, for instance, by evaluating
on $\unit$ the Möbius transform of the function $x\mapsto t^{\abs{x}}$\,.
From the first line of Table~\ref{tab:mobiusB4}, we read:
\[
  H_4(t)=(1-t)(1-2t-t^2+t^3+t^4+t^5)
\]

Let $p=q_4$ be the smallest root of~$H_4(t)$. We illustrate on the example
$x=12$ the computation of a line $P_{x,\text{\tiny$\bullet$}}$ of the
transition matrix corresponding to the uniform measure at infinity. From
Table~\ref{tab:mobiusB4}, we read the list of non zero values of the
corresponding line of the matrix, which are for this case: $2$, $21$, $23$,
$213$ and~$2132$. According to Theorem~\ref{thr:2}, for $x$ fixed, the
entries $P_{x,y}$ of the matrix are proportional to~$h(y)$, and the
normalization factor is $p^{-\abs x}h(x)$. Reading the values of $h$ in
Table~\ref{tab:mobiusB4}, we use the relation $1-2p-p^2+p^3+p^4+p^5=0$ to
write the coefficients as polynomials in~$p$, yielding:
\begin{align*}
  P_{12,2}&=p\,,
  &P_{12,213}&=-1+p+2p^2+2p^3+p^4\,,\\
  P_{12,21}=P_{12,23}&=p(1-2p^2-2p^3-p^4)\,,
  &P_{12,2132}&=p^4    \,.
\end{align*}

\subsection{Computations for $\dbr 4$}
\label{sec:computations-dbr-4}

We now treat the case of the $\dbr 4$. In order to simplify the notations, we
write $(ij)$ for the generator~$\sigma_{i,j}$\,, so for instance:
$\delta_4=(12)\cdot(23)\cdot(34)$\,.  The monoid $\dbr 4$ has the six
generators $(ij)$ for $1\leq i<j\leq4$, subject to the following relations:
\begin{align*}
  (12)\cdot(23)&=(23)\cdot(13)=(13)\cdot(12)
&(12)\cdot(24)&=(24)\cdot(14)=(14)\cdot(12)\\
(13)\cdot(34)&=(34)\cdot(14)=(14)\cdot(13)
&(23)\cdot(34)&=(34)\cdot(24)=(24)\cdot(23)\\
(12)\cdot(34)&=(34)\cdot(12)
&(23)\cdot(14)&=(14)\cdot(23)
\end{align*}

The set of simple braids $\SS_4$ has $14$ elements, which we organize below
according to the type of partition of the integer $4$ that the associated
non-crossing partition of $\{1,2,3,4\}$ induces (see
Subsection~\ref{sec:comb-epr-simple}):
\begin{align*}
  \SS_4=&\bigl\{\unit,&&\text{$1$ partition of type $1+1+1+1$}\\
&\delta_4,
&&\text{$1$ partition of type $4$}\\
&(12),\;(13),\;(14),\;(23),\;(24),\;(34),&&\text{$6$ partitions of type $2+1+1$}\\
&(12)\cdot(23),\;(12)\cdot(24),\;(13)\cdot(34),\;(23)\cdot(34),
&&\text{$4$ partitions of type $3+1$}\\
&(13)\cdot(24),\;(12)\cdot(34)\bigr\}
&&\text{$2$ partitions of type $2+2$}\\
\end{align*}

Following the same scheme as for~$\dbr 3$, we gather in
Table~\ref{tab:oijaoijaopakla} the characteristic elements for~$\dbr 4$\,. In
particular, the first line gives the Möbius polynomial, from which the
characteristic value $q_4$ is derived:
\begin{align*}
  H_4(t)&=(1-t)(1-5t+5t^2)\,,
&q_4&=\frac12-\frac{\sqrt5}{10}\,.
\end{align*}

\begin{table}
  \centering
\begin{gather*}
  \begin{array}{clll}
    x\in\SS_4&\{y\in\SS_4\tq x\to y\}&h(x)&\rho(x)\\
\hline
\rule{0em}{1em}    \unit&\unit&1-6p+10p^2-5p^3&0\\%&0\\
    \delta_4&  \SS_4     &p^3&1/5-2\sqrt5/25\\
    ({12})&(12),\;(13),\;(14)
                            &p(1-3p+2p^2)&\sqrt5/25\\%&1/10+1/2\sqrt5\\
    ({13})&(13),\;(14),\;(23),\;(24),\;(14)\cdot(23)
                            &p(1-2p+p^2)&1/10+\sqrt5/50\\%&3/10+1/2\sqrt5\\
    ({14})&(14),\;(24),\;(34)
                            &p(1-3p+2p^2)&\sqrt5/25\\%&1/10+1/2\sqrt5\\
    ({23})&(12),\;(23),\;(24)
                            &p(1-3p+2p^2)&\sqrt5/25\\%&1/10+1/2\sqrt5\\
    ({24})&(12),\;(13),\;(24),\;(34),\;(12)\cdot(34)
                            &p(1-2p+p^2)&1/10+\sqrt5/50\\%&3/10+1/2\sqrt5\\
    ({34})&(13),\;(23),\;(34)
                            &p(1-3p+2p^2)&\sqrt5/25\\%&1/10+1/2\sqrt5\\
    ({12})\cdot({23})&\SS_4\setminus\bigl\{\delta_4,\;(34),\;(23)\cdot(34),\;(13)\cdot(34),\;(12)\cdot(34)\bigr\}
                            &p^2(1-p)&1/10-\sqrt5/50\\%&1/2+1/2\sqrt5\\
    ({12})\cdot ({24})&\SS_4\setminus\bigl\{\delta_4,\;(23),\;(12)\cdot(23),\;(23)\cdot(34),\;(14)\cdot(23)\bigr\}
                            &p^2(1-p)&1/10-\sqrt5/50\\%&1/2+1/2\sqrt5\\
    ({23})\cdot ({34})&\SS_4\setminus\bigl\{\delta_4,\;(14),\;(12)\cdot(24),\;(13)\cdot(34),\;(14)\cdot(23)\bigr\}&p^2(1-p)&1/10-\sqrt5/50\\%&1/2+1/2\sqrt5\\
    ({13})\cdot ({34})&\SS_4\setminus\bigl\{\delta_4,\;(12),\;(12)\cdot(23),\;(12)\cdot(24),\;(12)\cdot(34)\bigr\}
                            &p^2(1-p)&1/10-\sqrt5/50\\%&1/2+1/2\sqrt5\\
    ({14})\cdot ({23})&
                                                  \SS_4\setminus\bigl\{\delta_4,\;(13),\;(12)\cdot(23),\;(13)\cdot(34)           \bigr\}
                            &p^2(1-p)&1/10-\sqrt5/50\\%&1/2+1/2\sqrt5\\
    ({12})\cdot ({34})& \SS_4\setminus\bigl\{\delta_4,\;(24),\;(12)\cdot(24),\;(23)\cdot(34)     \bigr\}
                            &p^2(1-p)&1/10-\sqrt5/50%&1/2+1/2\sqrt5\\
  \end{array}
\end{gather*}
\caption{Characteristic elements for the dual braid monoid~$\dbr
  4$. The column $h(x)$ gives the Möbius transform on $\SS_4$ of the
  function $f(x)=p^{|x|}$\,. The  column $\rho(x)$ evaluates the same
  quantity for the particular case $p=q_4$\,.}
  \label{tab:oijaoijaopakla}
\end{table}

The computation of the transition matrix of the chain of simple braids
induced by the uniform measure at infinity yields the values reported in
Table~\ref{tab:ojkazaaq}, where the line corresponding to $\delta_4$ has been
omitted. This line, which also corresponds to the initial measure of the
chain, is given by the function $\rho(x)$ tabulated in
Table~\ref{tab:oijaoijaopakla}.

\begin{table}
\small
  \begin{gather*}
\begin{aligned}
\left.\begin{array}{CCCCCCc}
  ({12})&
    ({13})&
    ({14})&
    ({23})&
    ({24})&
    ({34})&
\cdots
\end{array}\right.
\\
    \begin{array}{c}
  ({12})\\
    ({13})\\
    ({14})\\
    ({23})\\
    ({24})\\
    ({34})\\
    ({12})\cdot({23})\\
    ({12})\cdot ({24})\\
    ({23})\cdot ({34})\\
    ({13})\cdot ({34})\\
    ({14})\cdot ({23})\\
    ({12})\cdot ({34})
    \end{array}
\left(\begin{array}{CCCCCCc}
1/2-\rc&2\rc&1/2-\rc&0&0&0&\cdots\\%&0&0&0&0&0&0\\ %(12)
0&1/2-\rc&-1/2+3\rc&-1/2+3\rc&1/2-\rc&0&\cdots\\%&0&0&0&0&1-22\rc&0\\ %(13)
0&0&1/2-\rc&0&2\rc&1/2-\rc&\cdots\\%&0&0&0&0&0&0\\    % (14)
1/2-\rc&0&0&1/2-\rc&2\rc&0&\cdots\\%&0&0&0&0&0\\   % (23)
-1/2+3\rc&1/2-\rc&0&0&1/2-\rc&-1/2+3\rc&\cdots\\%&0&0&0&0&0&1-22\rc\\   % (24)
0&2\rc&0&1/2-\rc&0&1/2-\rc&\cdots\\%&0&0&0&0&0&0\\   % (34)
-1/10+\rc&1/5&-1/10+\rc&-1/10+\rc&1/5&0&\cdots\\%&3/10-\rc&3/10-\rc&0&0&3/10-\rc&0\\   % (12)(23)
-1/10+\rc&1/5&-1/10+\rc&0&1/5&-1/10+\rc&\cdots\\%&0&3/10-\rc&0&3/10-\rc&0&3/10-\rc\\   % (12)(24)
-1/10+\rc&1/5&0&-1/10+\rc&1/5&-1/10+\rc&\cdots\\%&3/10-\rc&0&3/10-\rc&0&0&3/10-\rc\\   % (23)(34)
0&1/5&-1/10+\rc&-1/10+\rc&1/5&-1/10+\rc&\cdots\\%&0&0&3/10-\rc&3/10-\rc&3/10-\rc&0\\   % (13)(34)
-1/10+\rc&0&-1/10+\rc&-1/10+\rc&1/5&-1/10+\rc&\cdots\\%&0&3/10-\rc&3/10-\rc&0&3/10-\rc&3/10-\rc\\   % (14)(23)
-1/10+\rc&1/5&-1/10+\rc&-1/10+\rc&0&-1/10+\rc&\cdots\\%&3/10-\rc&0&0&3/10-\rc&3/10-\rc&3/10-\rc\\   % (12)(34)
\end{array}\right.\\
\end{aligned}
\\[1em]
\begin{aligned}
&\left.    \begin{array}{cCCCCCCc}
\cdots&
    ({12})\cdot({23})&
    ({12})\cdot ({24})&
    ({23})\cdot ({34})&
    ({13})\cdot ({34})&
    ({14})\cdot ({23})&
    ({12})\cdot ({34})
    \end{array}
\right.
\\
&
\left.
    \begin{array}{cCCCCCCc}
\cdots&      0&0&0&0&0&0&\\ %(12)
\cdots&      0&0&0&0&1-4\rc&0&\\ %(13)
\cdots&      0&0&0&0&0&0&\\    % (14)
\cdots&      0&0&0&0&0&0&\\   % (23)
\cdots&      0&0&0&0&0&1-4\rc&\\   % (24)
\cdots&      0&0&0&0&0&0&\\   % (34)
\cdots&      3/10-\rc&3/10-\rc&0&0&3/10-\rc&0&\\   % (12)(23)
\cdots&      0&3/10-\rc&0&3/10-\rc&0&3/10-\rc&\\   % (12)(24)
\cdots&      3/10-\rc&0&3/10-\rc&0&0&3/10-\rc&\\   % (23)(34)
\cdots&      0&0&3/10-\rc&3/10-\rc&3/10-\rc&0&\\   % (13)(34)
\cdots&      0&3/10-\rc&3/10-\rc&0&3/10-\rc&3/10-\rc&\\   % (14)(23)
\cdots&      3/10-\rc&0&0&3/10-\rc&3/10-\rc&3/10-\rc&\\   % (12)(34)
    \end{array}
\right)
\begin{array}{c}
  ({12})\\
    ({13})\\
    ({14})\\
    ({23})\\
    ({24})\\
    ({34})\\
    ({12})\cdot({23})\\
    ({12})\cdot ({24})\\
    ({23})\cdot ({34})\\
    ({13})\cdot ({34})\\
    ({14})\cdot ({23})\\
    ({12})\cdot ({34})
    \end{array}
  \end{aligned}
\end{gather*}
\caption{Transition matrix for the uniform measure at infinity
  for~$\dbr 4$, restricted to its unique ergodic component
  $\SS_4\setminus\{\unit,\delta_4\}$, and where we have put $\theta=\sqrt5/10$.}
  \label{tab:ojkazaaq}
\end{table}

\section{Extensions}\label{se-ext}

There are various other questions of interest concerning the asymptotic
behavior of random braids, uniformly distributed among braids of length~$k$.
For instance, what is the asymptotic value of the height of a large braid? In
other words, how do the Garside length and the Artin length compare to each
other for large braids?

The height of braids gives rise to a sequence of integer random variables
$\height_k:\bbr n(k)\to\bbN$, indexed by~$k$, where $\bbr n(k)=\{x\in\bbr
n\tq\abs x=k\}$ is equipped with the uniform distribution. Since the ratios
height over length are uniformly bounded and bounded away from zero,
performing the correct normalization leads to considering the sequence of
real random variables $\rho_k:\bbr n(k)\to\bbR$ defined by:
\begin{gather*}
\forall x\in\bbr n(k)\qquad  \rho_k(x)=\frac{\height(x)}{\abs x}=\frac{\height(x)}k\,,
\end{gather*}
which takes values in the fixed interval $[1/\abs{\Delta},1]$. Since all
these random variables are defined on different probability spaces, the
natural way of studying their asymptotic behavior is by studying their
convergence in distribution.

A first result one may wish to establish is a \emph{concentration
  result}: one aims to prove that $(\rho_k)_{k\geq1}$ converges in
distribution toward a single value, say~$\rho$. Hence one expects a
convergence in distribution of the following form, where $\delta_\rho$
denotes the Dirac probability measure on the singleton~$\{\rho\}$:
\begin{gather}
  \label{eq:27}
\frac{\height(\cdot)}k\xrightarrow[k\to\infty]{\mathrm{\quad d\quad}}\delta_\rho\,,
\end{gather}
where $\rho$ is some real number in the open interval $(1/\abs{\Delta},1)$.
The number $\rho$ would appear as a \emph{limit
  average rate}: most of braids of Artin size $k$ would have, for $k$
large enough, a Garside size close to~$\rho k$. If $\rho$ can furthermore be
simply related to the quantities we have introduced earlier, such as the
characteristic parameter~$q_n$, it is reasonable to expect that $\rho$ would
be an algebraic number.

Once this would have been established, the next step would consist in
studying a Central Limit Theorem: upon normalization, is the distribution of
$\rho_k$ Gaussian around its limit value~$\rho$? Hence, one expects a
convergence in distribution of the following form, for some constant
$\sigma_n^2>0$ and where $\NN(0,\sigma^2)$ denotes the Normal distribution of zero
mean and variance~$\sigma^2$:
\begin{gather}
  \label{eq:21}
\sqrt
k\Bigl(\frac{\height(\cdot)}k-\rho\Bigr)\xrightarrow[k\to\infty]{\quad\mathrm{d}\quad}
\NN(0,\sigma_n^2)
\end{gather}

It turns out that both results~(\ref{eq:27}) and~(\ref{eq:21}) hold indeed.
Because of space constraints, we postpone their proofs to a forthcoming
work~\cite{opus2}.

This concerned an extension of the results established in this paper.
Generalizations to other monoids are also possible, which we intend to expose
in~\cite{opus2}. Braid monoids fall into the wider class of Artin-Tits
monoids, investigated by several authors since the 1960's, including Tits,
Deligne, Sato, Brieskorn, Garside, Charney, Dehornoy. Several results
established in this paper for braid monoids admit generalizations to
Artin-Tits monoids, and analogues of the convergences~(\ref{eq:27})
and~(\ref{eq:21}) also hold.

Among Artin-Tits monoids, one class in particular has retained the attention
of the authors: the class of trace monoids, also called partially commutative
monoids~\cite{cartier69}. In trace monoids, the only relations between
generators are commutativity relations (there is no braid relations); they
correspond to Viennot's \emph{heap monoids}~\cite{viennot86}. Trace monoids
differ from braid monoids for several reasons, for instance there is no
lattice structure and their associated Coxeter group is not finite. From the
point of view adopted in this paper, the main difference lies in the
existence of a \emph{continuum of multiplicative measures}, among which the
uniform measure is a particular case. Recall that we have observed in
Remark~\ref{rem:7} that the uniform measure for infinite braids is the only
instance of multiplicative measures, so the situation for braids presents a
sharp contrast with trace monoids. The investigation of multiplicative
measures for trace monoids has been the topic of~\cite{abbes15a}.

We shall prove in~\cite{opus2} that, from this perspective, there are
essentially only two types of Artin-Tits monoids: the trace type and the
braid type, corresponding respectively to the type with a continuum of
multiplicative measures, and the type where multiplicative measures reduce to
the uniform measure only. For the trace type, multiplicative measures are
parametrized by a sub-manifold of~$\bbR^m$, diffeomorphic to the standard
$(m-1)$-simplex, where $m$ is the minimal number of generators of the monoid.

\printbibliography

%\bibliography{biblio}

\end{document}